\documentclass[11pt, oneside]{amsart}

\usepackage{ amsmath, amssymb, amsthm, hyperref, graphicx,   verbatim, enumerate, color}
\usepackage{mathabx}
\usepackage[paper=a4paper]{geometry}
\usepackage[all]{xy}
\usepackage[utf8]{inputenc}

\date{\today}


\title{Degenerate homoclinic bifurcations in   complex dimension 2}

\author{Romain Dujardin}
\address{Sorbonne Universit\'e, Laboratoire de probabilit\'es, statistique et mod\'elisation, UMR 8001,  4 place Jussieu, 75005 Paris, France}
\email{romain.dujardin@sorbonne-universite.fr}

\newcommand{\C}{\mathbb{C}}
\newcommand{\B}{\mathbb{B}}
\newcommand{\R}{\mathbb{R}}
\newcommand{\D}{\mathbb{D}}
\newcommand{\Z}{\mathbb{Z}}
\newcommand{\N}{\mathbb{N}}
\renewcommand{\P}{\mathbb{P}}
\newcommand{\e}{\varepsilon}
\newcommand{\fkm}{\mathfrak{M}}

\newcommand{\fr}{\partial}

\newcommand{\set}[1]{\left\{#1\right\}}
\newcommand{\norm}[1]{\left\Vert#1\right\Vert}
\newcommand{\smallnorm}[1]{\Vert #1\Vert}
\newcommand{\abs}[1]{\left\vert#1\right\vert}
\newcommand{\cd}{{\C^2}}

\newcommand{\pu}{{\mathbb{P}^1}}

\newcommand{\rest}[1]{ \arrowvert_{#1}}

\newcommand{\unsur}[1]{\frac{1}{#1}}
\newcommand{\el}{\mathcal{L}}

\newcommand{\cst}{\mathrm{C}^\mathrm{st}}
\newcommand{\lrpar}[1]{\left(#1\right)}

\newcommand{\la}{\lambda}
\newcommand{\lo}{{\lambda_0}}
\newcommand{\loc}{{\mathrm{loc}}}
\newcommand{\hot}{{\mathrm{h.o.t.}}}

\newcommand{\La}{\Lambda}

\newcommand{\jstar}{J^\varstar}
\newcommand{\inv}{^{-1}}

\DeclareMathOperator{\mult}{mult}

\DeclareMathOperator{\codim}{codim}

\DeclareMathOperator{\jac}{Jac}

\DeclareMathOperator{\Diff}{Diff}
\DeclareMathOperator{\slope}{slope}

\newcommand{\diamant}{\medskip \begin{center}$\diamond$\end{center}\medskip}

\newtheorem{prop}{Proposition} [section]

\newtheorem{lem}[prop] {Lemma}

\newtheorem{propdef}[prop]{Proposition-Definition}

\theoremstyle{remark}

\newtheorem{rmk}[prop]{Remark}

\theoremstyle{definition}
\newtheorem{normalization}[prop]{Normalization}

\theoremstyle{theorem}
\newtheorem{mthmA}{Theorem}
\newtheorem{mthmB}{Theorem}
\newtheorem{mcorA}[mthmA]{Corollary}
\newtheorem{mcorB}[mthmB]{Corollary}

\newcommand{\commentaire}[1]{}
 
\begin{document}

\begin{abstract}
Unfolding homoclinic tangencies is the main source of bifurcations in 2-dimensional (real or complex) dynamics. When studying this phenomenon, it is common to assume that tangencies are quadratic and unfold with positive speed. Adapting to the complex setting an argument of Takens, we show that  any 1-parameter family of 2-dimensional holomorphic diffeomorphisms unfolding an arbitrary non-persistent 
homoclinic  tangency contains such quadratic tangencies. Combining this with  recent  
results of Avila-Lyubich-Zhang and former results in collaboration with Lyubich, this yields the abundance of robust homoclinic tangencies in the bifurcation locus for complex Hénon maps. We also study bifurcations induced by families with persistent tangencies, which provide another approach to the complex Newhouse phenomenon.
\end{abstract}

 \maketitle 
 
 \setcounter{tocdepth}{1}
 \tableofcontents
 
\section{Introduction}

Bifurcation theory is the study of the mechanisms creating instability in smooth dynamics. For 
surface diffeomorphisms, the most basic such mechanism is the unfolding of a homoclinic tangency, and a long standing conjecture of Palis predicts that  homoclinic tangencies are the building block  of 
all bifurcations. Recall that a homoclinic tangency is a tangency between the stable manifold and the unstable manifold of a saddle periodic point.
We refer the reader to the  classical monograph of Palis and Takens~\cite{palis-takens}
for an introduction to this topic. 

When studying the unfolding of a homoclinic tangency, 
it is common to restrict to a 1-parameter  family (or more generally a finite dimensional parameter family) where the unfolding is ``as transverse as it can be'', namely that 
the tangency is quadratic and  detaches  with positive speed. Without these assumptions, the 
  analysis of the bifurcation becomes much more delicate --the situation is somehow parallel to the difference between the quadratic family and a general multimodal family in one-dimensional dynamics. 
In the smooth ($C^k$) category this restriction is essentially harmless since one can always ensure 
these properties  in a generic family --hence the usual terminology 
\emph{generic homoclinic tangencies}. On the other hand, 
going to the analytic, or even algebraic, category, 
where the parameter spaces are typically much smaller, the genericity of such tangencies becomes 
an interesting problem. In~\cite{tangencies}, the prevalence of homoclinic tangencies in the space of complex Hénon mappings of a given degree was studied by M. Lyubich and the author, and the Palis conjecture was confirmed under mild dissipativity assumptions. 
However, the question of the genericity of these tangencies was left open.  

In a remarkable, but seemingly not so well-known,
 paper, Takens~\cite{takens} proved that in any family of 
real-analytic surface diffeomorphisms presenting an ``inevitable tangency'' (that is, in  which 
 a tangency must happen for topological reasons), then under a non-degeneracy assumption on the 
 eigenvalues at the saddle point,  
 generic tangencies are dense in the tangency locus. Other relevant references include 
 Robinson~\cite{robinson} and Davis~\cite{davis},  
 where cascades of sinks are created from  tangencies of 
  arbitrary order for real-analytic diffeomorphisms of surfaces (in~\cite{robinson}) and for $C^\infty$-
  diffeomorphisms under a $C^\infty$ linearizability condition (in~\cite{davis}).
 Many of the techniques developed in these papers take advantage of   plane topology and  
 geometry,   so non-trivial work needs  to be carried out to  adapt them to the complex setting. 
  Note that conversely, obtaining degenerate tangencies from non-degenerate ones  is also interesting 
  and leads to rich dynamical phenomena (see e.g. \cite{gonchenko_turaev_shilnikov, gonchenko_turaev_shilnikov2}).

\commentaire{Expo. Rmk \#8: paragraph added}
Several important  
references on homoclinic tangencies for 2-dimensional holomorphic diffeomorphisms deal with 
the problem of  the existence and prevalence of robust tangencies
(\footnote{We use the terminology  ``persistent homoclinic tangency'' only for a persistent tangency between the stable and unstable 
manifold of some \emph{periodic} point, and the adjective ``robust'' for Newhouse-type tangencies 
associated to hyperbolic sets.}), which in the dissipative case yields locally dense set of parameters 
  with infinitely many sinks  (the Newhouse phenomenon). Buzzard~\cite{buzzard:annals} showed the 
  existence of robust homoclinic tangences in the space of  polynomial automorphisms of sufficiently 
  large degree. 
 In relation with this work,  Biebler~\cite{biebler:gap} and Araujo-Moreira~\cite{araujo-moreira}    
 studied geometric   mechanisms for the  robust intersection  of plane Cantor sets.
  Very recently, Avila, Lyubich and Zhang~\cite{avila-lyubich-zhang} have  
  announced that  for dissipative holomorphic diffeomorphisms, 
  every non-persistent  quadratic homoclinic tangency yields a robust one. This relies on a different mechanism for intersections of plane Cantor sets. 
  Finally, Berger and Biebler~\cite{berger-biebler}   used real polynomial diffeomorphisms  unfolding  
  5 independent tangencies to construct examples  of  wandering Fatou components. 

\diamant 

The first main result in this paper is a generalization of Takens' theorem 
 to   holomorphic diffeomorphisms. 

\begin{mthmA}\label{thm:quadratic}
Let $(f_\lambda)_{\lambda\in \Lambda}$ be a holomorphic family of holomorphic diffeomorphisms 
defined in some  domain $\Omega\subset \cd$,   parameterized by a complex manifold $\Lambda$. 
Assume that in the neighborhood of 
 some  $\lambda_0\in \Lambda$, $f_\la$ possesses a saddle fixed point $p_\la$ 
 with a non persistent homoclinic tangency at $\lo$.  Assume furthermore that there is no persistent relation of the form $u_\lambda^as_\lambda^b=  1$, where  
 $s_\lambda$ and $u_\lambda$ are the respective stable and unstable multipliers of $p_\lambda$ and $a$ and $b$ are positive integers. 

Then there exists $\lambda_1$ 
arbitrarily close to $\lo$  such that $p_{\lambda_1}$ has   
 a quadratic homoclinic tangency,  unfolding with positive speed. 
\end{mthmA}

Note that the notion of ``unfolding with positive speed'' really makes sense only if $\Lambda$ is 1-dimensional (see \S\ref{sec:complex_tangencies} for a thorough discussion). 
For higher dimensional families, this means that  there is a 
1-dimensional family through $\lo$ in which this property holds. Hence 
 the result is strongest when $\Lambda$ is 1-dimensional, and we will prove it in this case. We note that  in the quadratic case, the ``positive speed'' assertion of the theorem was also part of the announcement~\cite{avila-lyubich-zhang}. 
 
To understand the subtlety of this result we have to recall how secondary 
tangencies are produced from the unfolding of an initial homoclinic tangency.  The mechanism is of course very classical.  
Assume that $(f_\la)_{\la\in \La}$ is a family of local diffeomorphisms of $\cd$, and 
$p = (p_\la)$ is a fixed saddle point such that for $\la = \lo$, $W^s(p_\lo)$ is tangent to $W^u(p_\lo)$ 
at $\tau$. We can work in local coordinates $(x,y)$ where $p_\la=(0,0)$
$W^s_\loc(p_\la) = \set{x=0}$ and $W^u_\loc(p_\la)   = \set{y=0}$. Iterating $\tau$ if necessary, we may 
assume that it belongs to $W^s_\loc(p_\lo)$, so there is a branch $\Delta^u_\lo\subset W^u(\lo)$ 
tangent to $\set{x=0}$ at $\tau$. Assume that $p_\lo$ belongs to some horseshoe (in the complex case 
this is automatic, see Proposition~\ref{prop:secondary}), whose stable lamination accumulates 
$\set{x=0}$. When $\la$ moves in parameter space, the horseshoe persists and   the branch $\Delta^u_\la$ is pulled across  
 the stable lamination, so new tangencies are created (see Figure~\ref{fig:secondary}). 

\begin{figure}[h]\label{fig:secondary}
\includegraphics[width=14cm]{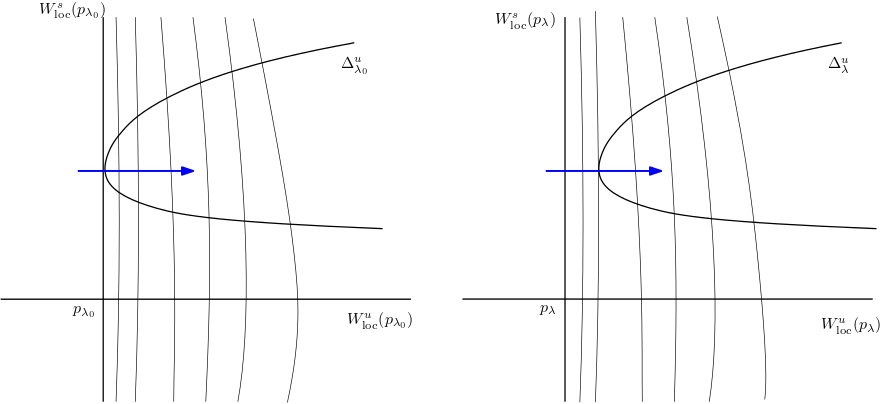}
\caption{Creation of a secondary tangency}
\end{figure}

Now imagine a toy model for this situation where the stable manifolds of the horseshoe are just vertical 
lines and $\Delta^u_\lambda$ moves under  a horizontal translation. 
In the holomorphic case one can easily imagine that the speed of motion of  
$\Delta^u_\lambda$ cannot vanish on a Cantor set of parameters (of course the reality is more 
complicated because the Cantor set of vertical lines moves with the parameter), so most tangencies should occur with positive speed. On the other hand, 
  on this toy model it is unclear why, if we start with a tangency of high order,
   the order of the secondary tangencies would generically decrease. 
The point of Takens' proof is to understand how the stable lamination of the horseshoe  
 \emph{differs} from a Cantor set of vertical lines, in order  to lower the order of tangency. In this respect, a key 
 notion is that of \emph{dynamical slope} (see \S\ref{subs:slope}).  The argument also requires delicate 
 $C^k$  estimates for these vertical graphs, for large $k$ (see \S\ref{subs:graph_transform}). 
The structure of the proof in the complex case is roughly the same as that of~\cite{takens}, 
but the  technical  details differ in many ways.  

\diamant

As a consequence of Theorem~\ref{thm:quadratic}, 
 all the phenomena associated to (one-dimensional) unfoldings of 
  generic homoclinic tangencies appear in $\Lambda$. In particular by
  \cite{avila-lyubich-zhang}, if $(f_\lambda)$ is dissipative, 
  $\Lambda$ contains Newhouse domains
  with robust homoclinic 
  tangencies  and residually 
  infinitely many sinks.
  
  Another consequence is an extension and a strengthening 
  of   McMullen's ``universality of the Mandelbrot set''~\cite{mcmullen-universal}: 
 applying the   quadratic renormalization theory of~\cite{palis-takens}, it follows that in any 
 one dimensional family $\Lambda$, 
 baby  Mandelbrot-looking 
 sets (contrary to~\cite{mcmullen-universal}, we do not have to deal with  Multibrot sets) 
 appear in parameter space near any homoclinic tangency (\footnote{These are not actual copies of the Mandelbrot set, since by the aforementioned results of~\cite{avila-lyubich-zhang}, the bifurcation locus has non-empty interior.}). 
 An interpretation of  Theorem~\ref{thm:quadratic} from the point of view of the analogy with 
 one-dimensional dynamics is that  active critical points of   higher order do not exist   for 2D diffeomorphisms (at least, under a non-resonance assumption).

For polynomial automorphisms  of $\C^2$ (i.e. generalized  complex Hénon maps), 
we can get rid of the non-resonance assumption, 
at the expense of potentially choosing another periodic point. \commentaire{expo. rmk \#6: theorem A.2 becomes a corollary}
 
\begin{mcorA}\label{cor:quadratic_henon}
Let $(f_\lambda)_{\lambda\in \Lambda}$ be a holomorphic family of  
dissipative polynomial  
automorphisms  of $\C^2$ of constant dynamical degree,
 parameterized by a complex manifold $\Lambda$. 
Assume that $f_\lo$  admits a non-persistent homoclinic tangency. 
Then there exists $\lambda_1\in \Lambda$ 
arbitrarily close to $\lo$  such that $f_{\lambda_1}$ has   
 a quadratic homoclinic tangency,  unfolding with positive speed. 
\end{mcorA}
 
If $\Lambda$ is an open subset of 
 the space of all generalized Hénon maps of a given degree, by 
 Buzzard-Hruska-Illyashenko~\cite[Thm 1.4]{BHI} there is no persistent resonance between the multipliers of a given periodic point, 
 so Theorem~\ref{thm:quadratic} applies directly.  In particular no dissipativity assumption is required in this case. 
Let us also point out  that a weaker version of this result  was recently established 
in~\cite[Appendix]{araujo-moreira}, in which $f_{\lo}$ is perturbed in an infinite dimensional space of entire mappings.

 For families of polynomial diffeomorphisms of $\cd$, 
it was shown in~\cite{tangencies}  that in the moderately dissipative regime 
$\abs{\jac(f)} < \deg(f)^{-2}$,  homoclinic tangencies are dense in the bifurcation locus.   
Corollary~\ref{cor:quadratic_henon} thus implies that these homoclinic tangencies 
can be chosen to be  quadratic   with positive speed. 
Putting  this together with the forthcoming results of 
Avila-Lyubich-Zhang~\cite{avila-lyubich-zhang} we obtain:

\begin{mcorA}\label{cor:ALZ}
In any holomorphic family of moderately dissipative polynomial  diffeomorphisms of $\C^2$ of a given degree, the bifurcation locus is the closure of its interior. 
\end{mcorA}

\diamant

The second main result  of the paper is  that  for families of 
polynomial automorphisms of $\C^2$, persistent homoclinic tangencies also induce bifurcations.

\begin{mthmB}\label{thm:persistent}
Let $(f_\lambda)_{\lambda\in \Lambda}$ be a substantial family of  
polynomial   automorphisms  of $\C^2$ of constant dynamical degree,
 parameterized by a connected complex manifold $\Lambda$. Assume that there is 
 a  persistent homoclinic tangency associated to some saddle periodic point $p$ 
(with multipliers $u$ and $s$) such that 
 the function $\lambda \mapsto \frac{\ln \abs{u_\lambda}}{\ln \abs{s_\lambda}}$ is non-constant. 
Then $(f_\lambda)$ is not weakly $J^\varstar$-stable. 
\end{mthmB}

We refer to the Appendix for 
  the meaning of the word ``substantial'' and the notion of weak $\jstar$-stability from~\cite{tangencies}, which is a weak form of 
  structural stability on the Julia set. Here we content ourselves with 
pointing out  that any dissipative family is substantial by definition, and that in this case the 
failure of weak $\jstar$-stability means that some saddle bifurcates to a sink (which implies the creation of new tangencies when $\abs{\jac(f)} < \deg(f)^{-2}$). 

Since this result applies to any open subset of $\Lambda$, in the dissipative regime 
 it follows from~\cite[Cor. 4.5]{tangencies} 
that Newhouse parameters, that is parameters  displaying  infinitely many sinks,  
 are dense in $\Lambda$. Thus we obtain  an 
alternate approach to the existence of such parameters
which does not involve stable intersections of Cantor sets
(see   Yampolsky-Yang~\cite{yampolsky-yang} for yet another approach, also using~\cite{tangencies}).
 
It is quite simple to find examples of families satisfying the assumptions of 
Theorem~\ref{thm:persistent}. For instance, it is classical that 
in the space of quadratic Hénon mappings with parameters 
$(a,c)\in \C^\varstar\times \C$, 
$f_{a, c}(z,w) = (aw+z^2+c, az)$, any degenerate parameter of the form 
$(0, c)$ where $c$ is strictly post-critically finite can be continued (in infinitely many ways) 
as a 1-parameter family of Hénon maps   with a persistent tangency. 
It follows that $\set{0}\times \fr M $ (where $M$ is the Mandelbrot set)
lies in the closure of the set of Newhouse parameters. In this case  the assumption on the multipliers 
is easy to check because the Jacobian tends to zero along the  parameter curve, so $s_\lambda\to 0$ while 
$u_\lambda$ is bounded away from 0 and infinity.  
 
It is worth mentioning that  two such  curves (one homoclinic, one heteroclinic), landing at $(0, -2)$, 
were studied in detail  by Bedford and Smillie in~\cite{BS_R1, BS_R2}, in the real setting. Along these families, the Julia set $\jstar$ is contained in $\R^2$, and no sink nor  additional tangency is created as the Jacobian varies. This shows that Theorem~\ref{thm:persistent} is really about complex parameters. 
 
To get further and prove the abundance of Newhouse parameters in the bifurcation locus,
 we have to check that the assumption on the multipliers is generically satisfied, up to a change of periodic point. This is similar in spirit to Corollary~\ref{cor:quadratic_henon}, 
 but also a lot more delicate, and requires to work in the space of all 
 polynomial automorphisms of a given degree.
 
\begin{mthmB}\label{thm:newhouse_henon}
Let $\Lambda$ be an irreducible component of the space of generalized 
Hénon mappings of degree $d\geq 2$, and $\lo\in \Lambda$ be a   
parameter displaying a  homoclinic tangency. Then in any neighborhood of
 $\lo$ there is a    hypersurface $\Lambda_1\subset \Lambda$ 
 satisfying the assumptions of Theorem~\ref{thm:persistent}. 
 In particular, if $\abs{\jac f_\lo}\leq 1$,   $\lo$ belongs to 
  the closure of the set of Newhouse parameters.
\end{mthmB}

To prove this, we have to rule out the   unlikely phenomenon that as soon as a 
tangency is created  near $\lo$, then,
 along the corresponding hypersurface where this new tangency 
persists, the non-resonance condition of Theorem~\ref{thm:persistent} fails. 
Even if  such a coincidence is hardly plausible, excluding  it requires some non-trivial  
arguments. In particular we make heavy    use the genericity 
 results of~\cite{BHI}. The second statement of the theorem is of independent interest: it asserts that
  if a tangency occurs at a conservative parameter, it can always be perturbed to a dissipative one (see Lemma~\ref{lem:conservative}). \commentaire{cf. Major remark \#3}

Together with~\cite{tangencies},  Theorem~\ref{thm:newhouse_henon}
 gives an alternate argument  for the following weak
version of Corollary~\ref{cor:ALZ}. 

\begin{mcorB} \label{cor:newhouse_henon}
In the   family of all moderately dissipative polynomial  diffeomorphisms of $\C^2$ of a given degree, 
the bifurcation locus is contained in 
the closure of the  set of Newhouse parameters. 
\end{mcorB}
 
This result is     weaker than Corollary~\ref{cor:ALZ} because,   instead of   
open sets where Newhouse parameters are residual, it only provides codimension 1 laminations. 
On the other hand, the proof is  much  simpler since it does not 
resort to the results of~\cite{avila-lyubich-zhang} (which are not yet available in print). 
We also note that  
another approach to the construction of 
Newhouse parameters was  devised by  
Martens, Palmisano and Tao~\cite{martens-palmisano-tao}.

It is natural to wonder whether Theorem~\ref{thm:persistent} admits a local version as in Theorem~\ref{thm:quadratic}. The statement would be: \emph{in any family of diffeomorphisms of $\Omega\subset\cd$ with a persistent homoclinic tangency, and such that 
$\frac{\ln \abs{u_\lambda}}{\ln \abs{s_\lambda}}$ is non-constant,    non-persistent tangencies are created}. There is a simple mechanism for this in the real setting, 
going back to 
the work of Gavrilov and Shil'nikov~\cite{gavrilov-shilnikov},
which depends on the signs of certain auxiliary parameters (such restrictions are
 necessary in view of the above mentioned examples of Bedford-Smillie~\cite{BS_R2}).

To prove Theorem~\ref{thm:persistent}, we take a different path 
and use the notion  of \emph{moduli of stability},  
introduced by Palis~\cite{palis_moduli} and further developed  e.g. by
  Newhouse, Palis and Takens~\cite{newhouse-palis-takens} and also by
Buzzard~\cite{buzzard_nondensity} for diffeomorphisms of $\C^2$. 
In all these references, the authors start with a 
topological conjugacy between two diffeomorphisms 
 in a neighborhood of an orbit of tangency to  deduce a differentiable rigidity of the multipliers (\cite{palis_moduli} further relies on plane topology considerations). 
 We show that this notion  can be adapted to the  context of the weak $J^\varstar$-stability 
 theory of~\cite{tangencies}, which does \emph{not} yield   topological conjugacies. As observed above, the results of~\cite{BS_R1, BS_R2} 
 show that   it is essential here to work in the complex setting. 
 A key idea is that the 
 holomorphic motion of saddle periodic points admits a natural extension to stable and unstable manifolds, which satisfy good distortion properties.  
 This provides a reasonably  simple proof of  Theorem~\ref{thm:persistent}, which takes 
 advantage of the global geometric structure of complex Hénon mappings and showcases the techniques of~\cite{tangencies}. 
  
Note that   it  should  also be 
possible to adapt the Gavrilov-Shil'nikov  mechanism to the complex setting, 
at least for quadratic tangencies. We plan to come back to this issue in a later work.  

\diamant 

\commentaire{Expo. Rmk \#7: Former Remark 1.1 expanded to a subsection of the introduction}
 Our results  bear  some    similarity with the 
 recent work of Astorg and Bianchi~\cite{astorg-bianchi} and 
 Gauthier, Taflin and Vigny~\cite{GTV}, 
where  ``higher bifurcations''  are studied 
 in spaces of regular endomorphisms of $\mathbb{P}^k(\C)$ (see also~\cite{survey_icm} for a brief account on this topic). 

More precisely, Proposition 5.2 in~\cite{astorg-bianchi} asserts that in a family of 2-dimensional polynomial skew products   over a fixed base, admitting a persistent Misiurewicz relation (that is a relation of the form $f^n(c) = p$ where $c$ is critical and $p$ is a repelling periodic point) and satisfying appropriate  tranversality conditions, then secondary Misiurewicz bifurcations are created at every parameter. Thus, Theorem~\ref{thm:persistent} is an analogue of this result for homoclinic bifurcations.

The construction in~\cite{GTV} is more complicated, and combines Misiurewicz and homoclinic mechanisms. A common point with our work is the use of smooth linearization  at saddle points to get a precise control of iterated graphs. For this, the authors use the results of Sell~\cite{sell} on $C^1$ linearization, while   Theorem~\ref{thm:quadratic}  
 requires more precise $C^k$ estimates for large $k$, for which we need to come back to the original work of Sternberg~\cite{sternberg2}. Another aspect of~\cite{GTV} is the verification of various transversality conditions in the parameter space, which is reminiscent of Theorem~\ref{thm:newhouse_henon}. 
 
More generally,  our results can 
be interpreted from the  perspective of higher bifurcations (at least in the moderately dissipative regime): 
 indeed the bifurcation  locus has some codimension 1 structure given by the hypersurfaces  
 of persistent homoclinic tangencies, and inside these hypersurfaces 
 Theorem~\ref{thm:persistent} 
   can recursively  be used to construct  new tangencies.
 One might expect that the analogue of Theorem~\ref{thm:newhouse_henon} holds 
 recursively, so that if $\La$ is a component of the space of polynomial automorphisms of given degree, 
 then  on a dense subset of the bifurcation locus there should be   $\dim(\La)$ 
 ``independent'' tangencies (compare with~\cite{berger-biebler}). 

\diamant

\subsection*{Outline}  We start in Section~\ref{sec:geometric_lemmas}
with some geometric preliminaries on submanifolds of the bidisk.  In 
Section~\ref{sec:complex_tangencies}, we  apply basic 
 ideas from  local complex geometry 
 to define and give a neat treatment of the notions of 
order, speed exponent and multiplicity of a non-persistent tangency. 
Theorems~\ref{thm:quadratic} 
and Corollary~\ref{cor:quadratic_henon} are proven in Section~\ref{sec:quadratic}. 
As explained above, the most delicate point is to produce quadratic tangencies, which requires 
$C^k$ graph transform estimates (\S\ref{subs:graph_transform}) and to develop a notion of dynamical slope (\S\ref{subs:slope}). 
In Section~\ref{sec:persistent} we study persistent tangencies and prove Theorems~\ref{thm:persistent} 
and~\ref{thm:newhouse_henon}. In the  Appendix we briefly review the notion of weak stability from~\cite{tangencies}. 
  
\subsection*{Notation and conventions} 
The unit disk in $\C$ is denoted by $\D$ and we let $\B  = \D\times \D$. 
The letter $C$ stands for a ``constant" which may change from line to line,   independently of 
some asymptotic quantity that should be clear from the context.  
We make heavy use of the following notation: we write
 $a\lesssim b$ if $\abs{a}\leq C\abs{b}$, $a\asymp b$ if $a\lesssim b \lesssim a$, and 
 $a\cong b$ if $a\sim cb$ for some $c\in \C^\varstar$.
 We denote by $\norm{\cdot}_\Omega$ the uniform norm in a domain $\Omega$. 
The eigenvalues at a saddle periodic point will be generally denoted by $u$ and $s$ (or $u_\lambda$ and $s_\lambda$ in the presence of a parameter), where it is understood that $\abs{u}>1$ and $\abs{s}<1$.

\subsection*{Acknowledgments} Thanks to Artur Avila, Misha Lyubich and Zhiyuan Zhang for their 
work on the complex Newhouse phenomenon which prompted me to finally tackle this project,
  to Sébastien Biebler for  many  discussions on this topic, and to Marc Chaperon for his help 
  with Sternberg's theory. And special thanks to the anonymous referee whose 
  thoughtful comments greatly improved the paper.

\section{Preliminary geometric lemmas}\label{sec:geometric_lemmas}

\subsection{A transversality lemma}  In this paper, a holomorphic disk is an embedding of the unit disk into $\C^2$\commentaire{cf. Minor comment \#2}.    
The following basic lemma is very useful.

\begin{lem}[see {\cite[Lem. 6.4]{bls}}] \label{lem:6.4}
Let $\Delta$ and $\Delta'$ be two holomorphic disks with an isolated tangency of order $h$ at $p$. Then if $\Delta''$ is a holomorphic disk  disjoint from  $\Delta$ and sufficiently $C^1$ close to it, 
it intersects $\Delta'$ transversally in $h+1$ points close to $p$.
\end{lem}

Observe that this result does not hold in higher dimension, due to the possibility of non-proper intersections: this is precisely  the reason why Theorem~\ref{thm:quadratic} is non-trivial (cf. the formalism of \S\ref{subs:lifting}). 
An explicit example where the higher dimensional version of this lemma fails 
is obtained by lifting the basic toy model from the Introduction to the projectivized tangent bundle. 

\subsection{Horizontal and vertical varieties in $\B$}
Recall that $\B$ is the unit bidisk. We let $\fr^h\B=\D\times \fr\D$ (resp.  $\fr^v\B=\fr\D\times \D$) be its horizontal (resp. vertical) boundary. A subvariety $V$ in 
 some neighborhood of $\overline\B$  is \emph{horizontal} (resp. \emph{vertical}) in $\B$ if 
$V\cap \fr^h\B = \emptyset$ (resp. $V\cap \fr^v\B = \emptyset$). A horizontal (resp. vertical) subvariety is a branched cover over the unit disk for the first (resp. second) 
projection so it has a \emph{degree}, which is the degree of this cover.
A horizontal  (resp. vertical) graph is a horizontal (resp. vertical) subvariety of degree 1\commentaire{cf. Minor comment \#2}.   
 If $V_1$ (resp. $V_2$) is a horizontal (resp. vertical) variety of degree $d_1$ (resp. $d_2$), then  $V_1$ and $V_2$ intersect in $d_1d_2$ points, counting multiplicity
 (see e.g. \cite{henonl} for more details on these notions). 

\begin{lem}\label{lem:riemann-hurwitz}
Let $\Delta$ be a horizontal submanifold in $\B$ of degree $d$ 
which is a union of holomorphic disks, and $(W_i)_{i\in I}$ an arbitrary collection of disjoint vertical graphs. Then the total number of tangencies between $\Delta$ and the $W_i$, counting multiplicities, 
 is bounded by $d-1$
\end{lem}
 
\begin{proof}
The union of the $W_i$ and $\fr^v\B   = \bigcup_{\zeta\in \fr\D} \set{\zeta}\times \D$ is a lamination 
 by vertical graphs. By Lemma~\ref{lem:6.4}, 
 by slightly perturbing the radius of the bidisk, we may assume that 
$\Delta$ is transverse to $\fr^v\B$. If we fix a horizontal slice, for instance $L:=\D\times \set{0}$, 
 which we identify to $\D$, this lamination can be viewed as a holomorphic motion of 
 $\fr\D\cup (\bigcup W_i\cap L)$, which can be extended to a motion of $\overline \D$
by Slodkowski's theorem~\cite{slodkowski}. Note that since the motion is the identity on $\fr\D$, 
it must preserve $\D$. 
The corresponding lamination   of $\B$ by vertical graphs  fills up the whole 
bidisk. 

Let $\pi:\B\to\D$ be the projection along this lamination. Since $\Delta$ intersects any vertical graph in 
$d$ points, $\pi\rest{\Delta}:\Delta\to \D$ is a branched cover of degree $d$.
 If we can show that  this branched cover
 satisfies the Riemann-Hurwitz formula, then the total number of tangencies, counting 
multiplicity, is at most $(d-1)$, and the lemma follows. To prove this fact, we first note that by 
Lemma~\ref{lem:6.4}, only finitely many $W_i$ are tangent to $\Delta$. They correspond to the 
critical points of $\pi\rest{\Delta}$. 
Then we argue as in the 
usual proof of the Riemann-Hurwitz formula, 
by pulling back a triangulation of the base whose vertex set 
contains the critical values,   and computing the Euler characteristic of the pulled-back triangulation (which is equal to the number of components of $\Delta$). The only delicate point is to show that at the critical points, $\pi$ behaves topologically like a holomorphic map: this follows from the fact that near such a point, $\pi$ is of the form 
$u\circ k$, where $u$ is a holomorphic map and $k$ is a quasiconformal homeomorphism 
(see~\cite[p. 590]{approx}).  
\end{proof}

\begin{lem}\label{lem:vertical_tangencies}
Let $\Delta$ be a horizontal submanifold in $\B$ of degree $d$ 
which is a union of holomorphic disks. 
Assume that there is a vertical graph $W$  in $\B$ 
which is tangent to order $d-1$ to $\Delta$. 
Then $\Delta$ has $d$ vertical tangencies in $\B$
\end{lem}

\begin{proof}
Fix $R<1$ such that $W$ is vertical in $R\D\times \D$. 
Since the number of vertical tangencies of $\Delta$ is finite, there exists $R\leq R'\leq 1$ such that 
$\Delta$ is transverse to $\fr (R'\D)\times \D$. Without loss of generality we replace $\D$ by $R'\D$. 
Each component of $\Delta$ is a horizontal submanifold, hence it intersects $W$. 
Since $W\cap \Delta$ admits  $d$ points counting multiplicities, we infer that 
$W\cap \Delta$ is reduced to the tangency point, of multiplicity $d$. Hence 
  $\Delta$ is made of a single component, and the result follows from the Riemann-Hurwitz formula applied to the first projection. 
\end{proof}

 \subsection{Tangency creation lemma}
 
\begin{lem}[see {\cite[Prop. 9.1]{tangencies}}]\label{lem:creating_tangencies}   
Let $(\Delta_\lambda)_{\lambda\in \Lambda}$ be a holomorphic family of horizontal submanifolds of 
degree $d$ in $\B$, with $\Lambda\simeq\D$. Assume that for $\lambda$ close to $\fr\Lambda$, $\Delta_\lambda$ is a union of horizontal graphs, and that  this property does not hold for some $\lambda_0\in \Lambda$. 

Let $(W_\lambda)_{\lambda\in \Lambda}$ be a holomorphic family of vertical graphs in $\B$. Then 
there exists a non-empty finite set of parameters such that $\Delta_\lambda$ and $W_\lambda$ are tangent. 
\end{lem}

The finiteness of the set of tangency parameters was not stated in 
 \cite[Prop. 9.1]{tangencies} but it is explicitly mentioned in its proof 
(see also~\cite{taflin-tangencies} for a counting of the number of tangency parameters).

\section{Complex geometry of homoclinic tangencies}\label{sec:complex_tangencies}

\subsection{Conventions}\label{subs:conventions} ~\commentaire{Lots of changes in this subsection, in preparation for \S 4.3 (cf. Exposition remark \#2}
Let $\Omega\subset \C^2$ be an open set and $f:\Omega \to f(\Omega)\subset \C^2$ be a diffeomorphism with a saddle fixed point $p$, with local stable and unstable manifolds 
$W^{s/u}_\loc(p)$. 
In such a semi-local setting, the global stable  
  manifold $W^{s}(p)$ is the set of points $x\in\Omega$ such that $f^{n}(x)$ belongs to $\Omega$ for every $n\geq 0$ and eventually falls into $W^{s}_\loc(p)$, and likewise for $W^u(p)$. 
  
Assume that there is a  tangency between $W^s(p)$ and $W^u(p)$ at $\tau$. 
We will generally
 use the following normalization: replacing the tangency point by an iterate  
 one may assume that $\tau\in W^s_\loc(p)$;
we pick local coordinates $(x,y)\in \B$ 
such that $p=0$, $f(x,y)  = (ux, sy)+ \hot$ and 
$W^s_\loc(p) = \set{x=0}$, so that $\tau= (0, y_0)$. We assume that $\B$ is small enough  so that $f$ uniformly expands the horizontal direction and contracts the vertical direction in $\B$ (in particular it is a Hénon-like mapping of degree 1 in the sense of \cite{henonl}).
The component  of $\tau$ in $W^u(p)\cap \B$, 
denoted $\Delta^u$, is locally of the form $x = \varphi(y)$, where 
$\varphi$  a holomorphic function defined in a neighborhood of $y_0$, with 
$\varphi(y_0) = \varphi'(y_0) = 0$ (see Figure~\ref{fig:secondary}).

We say that the tangency is \emph{of order $h$} if the order of contact --which must be  finite in the holomorphic setting-- 
equals $h+1$, that is  
$\varphi(y)\cong (y-y_0)^{h+1}$ as $y\to y_0$ (recall that this means $\varphi(y)\sim
c (y-y_0)^{h+1}$ for some $c\neq 0$).   So a quadratic tangency is a tangency of order 1. 

Since for large $n$, $f^{-n}(\B)\cap B$ is a thin tube around $W^s_\loc(p)$, the component of $\tau$ in 
  $\Delta^u\cap f^{-n}(\B)$ is contained in the graph of $\varphi$. It follows that replacing 
  $\Delta^u$ by the cut-off iterate $f^n(\Delta^u)\cap \B$, $\Delta^u$ is a horizontal disk in $\B$ 
  with $\Delta^u\cap W^s_\loc(p) = \set{\tau}$, hence its horizontal degree equals $h+1$.

Now let $f$ depend holomorphically on a parameter $\lambda\in \Lambda$, so that all the above defined 
objects depend holomorphically on $\lambda$, 
and are denoted by $f_\lambda$, $p_\lambda$, etc. 
In this context, the notation  $p$ stands 
 for the holomorphically moving  family  $(p_\lambda)_{\lambda\in \Lambda}$.
The parameter space $\Lambda$ is typically the unit disk,
but for clarity we keep the notation $\Lambda$. Again we   choose  local 
 coordinates $(x,y)$ such that $p_\lambda = (0,0)$ and $W^s_\loc(p_\lambda) = \set{x=0}$. 
The tangency parameter is denoted by $\lambda_0$, and without loss of generality
 we often assume $\lambda_0 = 0$. 

\begin{normalization}\label{norm:tang}\commentaire{We make the normalization more explicit and precise for future reference (notably in \S 4.3)}
There are local coordinates $(x,y)\in \B$ such that the 
 holomorphic family $(f_\lambda)_{\lambda\in \Lambda}$ and the parameter $\lo$
satisfy  the following properties :
\begin{itemize}
\item $p_\lambda = (0,0)$ is a saddle fixed point of $f_\lambda$.
\item $W^s_\loc(p_\lambda) = \set{x=0}$, $W^u_\loc(p_\lambda) = \set{y=0}$.
\item At  $\lambda = \lo$ there is a tangency between  $W^u(p_\lambda)$ 
and $W^s_\loc(p_\lambda)$ at $\tau = (0, y_0)$. 
\item For $\lambda$ close to $\lambda_0$, there is a component 
$\Delta^u_\lambda$ of $W^u(p_\lambda)\cap \B$,  
which is a horizontal disk of degree $h+1$  in $\B$ 
depending holomorphically  on $\lambda$. For  $\lambda = \lambda_0$,  the only intersection point 
between $\Delta^u_{\lambda_0} $ and $W^s_\loc(p_{\lambda_0})$ is $(0,y_0)$, and  
near $(0,y_0)$, 
$\Delta^u_\lambda$ is expressed as a graph of the form 
$x = \varphi(\lambda, y)$, where $\varphi(\lambda, \cdot)$ is a holomorphic function defined in a neighborhood of $y_0$ and depending holomorphically  on $\lambda$.   
\end{itemize}
\end{normalization}
 
Any family of semi-local holomorphic diffeomorphisms with a homoclinic tangency, 
can be put under this  form by reducing the 
domain of definition, choosing appropriate coordinates, 
and considering a suitable iterate of the tangency point, as explained above.  
Note that the last statement that $\Delta^u_\lambda$
 is locally a graph over the second coordinate 
is an immediate consequence of the existence of a tangency.

The tangency is \emph{non-persistent} if
$\Delta^u_\lambda$ is transverse  to $W^s_\loc(p_\lambda)$ for some $\lambda\in \Lambda$. 
Note that if $\Lambda$ is the unit disk, this property will then hold for every $\lambda\neq\lambda_0$   
close to $\lambda_0$. By the persistence of 
proper intersections, $\Delta^u_\lambda$ must have $h+1$  transverse intersections with  $W^s_\loc(p_\lambda)$ for  such  $\lambda$.   
 Note also that as soon as $\Delta^u_\lambda$ remains horizontal, by the persistence of the intersection number with a given vertical line,  its degree remains constant and equal to $h+1$.

It is not so easy in the smooth category 
to define a formal notion 
of ``speed of motion'' for the tangency: first, when $\lambda\neq\lambda_0$, there is no homoclinic
 tangency to work with. 
A common idea is to consider   vertical tangencies of $\Delta^u_\lambda$ as a kind  of 
 ``virtual'' tangency, and look at the speed of motion of these tangencies (\footnote{Then, an additional argument would be required to show that this notion is intrinsic, in a given regularity 
 class for  the family $(f_\lambda)$.}).
In the complex setting, when $h>1$ the tangency usually splits into $h$ vertical 
tangencies, which  may or may not be followed holomorphically. The good news is that there is 
still a unequivocal notion of a \emph{generically unfolding} or ``positive speed'' tangency, 
as we will explain in \S\ref{subs:speed} below.  

\begin{rmk}\label{rmk:semi-continuity-persistent}
Note that for persistent tangencies,  it may be the case 
that for $\lambda=\lo$ the tangency is of order $h$, while for $\lambda\neq \lo$ 
it splits off into a persistent tangency of order $h'<h$ together with $h - h'$ vertical tangencies.
\end{rmk}

\subsection{Lifting to the projectivized tangent bundle} \label{subs:lifting}
The following viewpoint was developed in \cite[\S 9.1]{tangencies}. Consider a family of two holomorphically varying smooth complex submanifolds 
$(V_\la)_{\lambda\in \Lambda}$ and $(W_\la)_{\lambda\in \Lambda}$ in $\B$, with $\Lambda = \D$. 
For every $\la$, we denote by $\P T V_\lambda$ (resp. $\P T W_\lambda$) the lift 
of $V_\lambda$ (resp. $W_\lambda$) to the projectivized tangent bundle $\P T\B \simeq \B \times \pu$. Finally, we define a 2-dimensional subvariety 
 $\widehat {\P T V}$ of $\Lambda\times \B \times \pu$ by putting together the $\P T V_\lambda$, that 
 is $\widehat {\P T V} \cap \set{\lambda}\times \B \times \pu = \P T V_\lambda$, and likewise 
 for $\widehat {\P T W}$. 
  A non-persistent tangency between $V_0$ and $W_0$ then corresponds to an isolated   intersection of $\widehat {\P T V}$ and $\widehat {\P T W}$. Let us denote by $\widehat 0$ the corresponding intersection point.  Since in this 
 case $$\codim_{\widehat 0} 
 \lrpar{\widehat {\P T V}\cap \widehat {\P T W}}  = \codim_{\widehat 0}  \lrpar{\widehat {\P T V}} + 
 \codim_{\widehat 0}  \lrpar{\widehat {\P T W}}, $$ this intersection is proper so it has nice properties, in particular the intersection multiplicity $\mult_{\widehat{0}}\lrpar{\widehat {\P T V}, \widehat {\P T W}}$ is well-defined (see \cite[\S 12]{chirka}). By definition this number is the \emph{multiplicity} of the tangency. 
 
\begin{lem}\label{lem:multiplicity}
With notation as in \S \ref{subs:conventions}, $\widehat{\P TW^s}$ and 
$\widehat{\P T\Delta^u}$ are smooth, and   
\begin{equation}\label{eq:multiplicity}
\mult_{\widehat{0}}\lrpar{\widehat {\P T W^s}, \widehat {\P T \Delta^u}}\geq h ,
\end{equation}
where $h$ is the order of tangency.  
\end{lem}

\begin{proof} 
The multiplicity in \eqref{eq:multiplicity} can be computed by intersecting two curves, as follows. Recall that we fixed local  coordinates such that $W^s_\loc(p_\lambda) = \set{x=0}$, and 
$\Delta^u_\lambda = \set{x= \varphi_\lambda(y)}$. Write $y=y_0+t$ and $\varphi_{y_0}(\lambda, t ) :=\varphi(\lambda, y_0+t)$ so that near the point of tangency, $\Delta^u_\lambda$ is parameterized by $(  \varphi_{y_0}(\lambda, t ),y_0+t )$, whose tangent vector 
is $(\dot{\varphi}_{y_0}(\lambda, t),1)$, 
where the dot denotes derivation with respect to the $t$-variable. In 
$\P T\C^2\simeq\P^1$, denote by $w$ the coordinate in 
 the affine chart containing $[0, 1]$. Thus we have local coordinates near $\widehat 0$ in
 $(\lambda, t, x, w)$   in 
 $\Lambda\times \B \times \pu$, in which we have the equations 
 $$\begin{cases} x= \varphi_{y_0}(\lambda, t) \\ w  = \dot{\varphi}_{y_0}(\lambda,  t) \end{cases} \text{ for } 
 \widehat {\P T \Delta^u} \ \text{ and } 
 \begin{cases} x=0 \\ w  = 0 \end{cases} \text{ for } 
 \widehat {\P T W^s},$$
 and it follows that both varieties are smooth near $\widehat 0$.   
Since $\Delta^u_\lambda$ is smooth and varies holomorphically, a similar argument  establishes  the  smoothness of  $\widehat {\P T \Delta^u}$ at every point. 
 
 We claim that    
   $  \mult_{\hat{0}}({\widehat {\P T W^s}, \widehat {\P T \Delta^u}})$ is equal to 
 the intersection multiplicity of the curves 
 $\mathcal C:= \set{\varphi_{y_0}(\lambda, t)  = 0}$ and $\dot{\mathcal {C}}:= 
 \set{\dot{\varphi}_{y_0}(\lambda, t) = 0}$, which 
 admit an isolated intersection at $(0,0)$. \commentaire{Precision added here}
 Indeed,  
  the multiplicity  is  the number of intersections of  generic local  translations of 
 $ \widehat {\P T \Delta^u}$ and $ \widehat {\P T W^s}$ (see \cite[\S 12]{chirka}), so it is the number of intersection points between 
   $$\begin{cases} x= \varphi_{y_0}(\lambda, t) +\e_1 \\ w  = \dot{\varphi}_{y_0}(\lambda,  t)  +\e_2 \end{cases}   \text{ and } 
 \begin{cases} x=\e_3 \\ w  = \e_4 \end{cases},$$ for generic $(\e_1, \e_2, \e_3, \e_4)$, which corresponds to  
 $$\set{x= \varphi_{y_0}(\lambda, t) +\e_1 - \e_3}\cap \set{w  = \dot{\varphi}_{y_0}(\lambda,  t) + \e_2  -\e_4}
 $$ 
 and the claim follows.

On the other hand, that the classical definition of multiplicity of the intersection of two curves 
 is  $\mult_0(\mathcal C, \dot{\mathcal C}) = 
 \dim\mathcal{O}(\cd, 0)/\langle \varphi_{y_0}, \dot \varphi_{y_0}\rangle$.
  In the situation at hand we have 
\begin{equation}\label{eq:gamma_gammaprime}
\begin{cases}
 \varphi_{y_0}(\lambda, t)   = c t^{h+1} + O({t^{h+2}})+ O\lrpar{\lambda} \\
 \dot\varphi_{y_0}(\lambda, t)   = ch t^{h} + O({t^{h+1}})+ O\lrpar{\lambda} \end{cases}
 \end{equation} hence 
$\mathcal{O}(\cd, 0)/\langle \varphi_{y_0}, \dot  \varphi_{y_0}\rangle$ contains all monomials 
$1, t , \ldots , t^{h-1}$ and its dimension is at least $h$.  
\end{proof}
  
\subsection{Speed of motion} \label{subs:speed}  
We now explain how  the multiplicity defined above  takes into account both the order 
and the speed of motion of the tangency. 
Recall from the   proof of Lemma~\ref{lem:multiplicity} 
that the multiplicity of tangency equals 
$\mult_0(\mathcal C, \dot{\mathcal C})$, where  
 $\mathcal C= \set{\varphi_{y_0}(\lambda, t)  = 0}$ and $\dot{\mathcal {C}}= 
 \set{\dot{\varphi}_{y_0}(\lambda, t) = 0}$.
 
For simplicity , let us first assume that $\dot{ \mathcal {C} }$ is irreducible at $(0, 0)$. For a
 fixed  small $\lambda\neq 0$, the equation $\dot\varphi_{y_0}(\lambda, t) = 0$ admits $h$ solutions 
 (that is, there are $h$ vertical tangencies)
  so  $\dot{\mathcal C}$ can be parameterized by a Puiseux series in $t^{1/h}$. 
 In usual complex geometric language, there is a coordinate $\mu$ on the normalization 
 of ${\dot{\mathcal C}}$ such that the expression of the composition of the 
 normalization map and the first projection $(\lambda, t)\mapsto \lambda$ is
 $\mu\mapsto \mu^h$, hence $\dot{\mathcal C}$ admits a  local (injective) 
 parameterization of the form 
  $\mu\mapsto (\mu^h, \theta(\mu))$, for some holomorphic function $\theta$ with 
 $\theta(0) = 0$.
 Writing 
 \begin{equation}
 \varphi_{y_0}(\lambda, t) = c t^{h+1} + O({t^{h+2}})+ \lambda \psi(\lambda, t)
 \end{equation}
  and 
 substituting, we infer that  
\begin{equation}\label{eq:gamma}
\varphi_{y_0}(\mu^h, \theta(\mu)) = c\theta (\mu)^{h+1} +  O({\theta (\mu)^{h+2}}) + 
 \mu^h \psi(\mu^h, \theta(\mu)),
 \end{equation}
 and the multiplicity $m$  of tangency  is the order of vanishing of this expression at $\mu = 0$ (since the tangency is non-persistent, $\mu\mapsto\varphi_{y_0}(\mu^h, \theta(\mu))$ has an isolated zero at the origin). 
 This yields another proof that $m\geq h$, with equality if and only if 
 $\psi(0, 0)\neq 0$, that is, 
\begin{equation}\label{eq:speed_order1}
\varphi_{y_0}(\lambda , t ) = ct^{h+1} + d \lambda  +\hot, \text{ with }  d=\psi(0,0). 
\end{equation}
 In the language of Puiseux series, this reads
\begin{equation}\label{eq:gamma_puiseux}
\varphi_{y_0}(\lambda, \theta(\lambda^{1/h})) = c\theta (\lambda^{1/h})^{h+1} +  
 O({\theta (\lambda^{1/h})^{h+2}}) + 
 \lambda \psi(\lambda , \theta(\lambda^{1/h})). 
 \end{equation}
The  speed of motion $\sigma$ of the vertical tangencies is characterized  by the 
 exponent of $\lambda$ in $\varphi_{y_0}(\lambda , \theta(\lambda^{1/h}))$, 
 that is, 
 \begin{equation}\label{eq:sigma}
 \varphi_{y_0}(\lambda , \theta(\lambda^{1/h}))\cong \lambda^\sigma
 \text{, where }
 \sigma = m/h\geq 1\end{equation} ($\sigma$  is typically not an integer), 
 and we see that \emph{non-vanishing speed of motion at $\lambda =0$ 
 (i.e. non-degenerate unfolding)
 corresponds   to $m=h$}, as expected from~\eqref{eq:speed_order1}.  
 Note that   for quadratic tangencies  ($h=1$), this simply 
means that \emph{$\widehat {\P T W^s}$ and $\widehat {\P T \Delta^u}$ are transverse at $\widehat 0$}. 
 
 Beware, however, that non-vanishing speed of motion (i.e. $\sigma = 1$) 
 does \emph{not} 
 imply that the vertical tangencies can be followed holomorphically, as shown for instance 
 by the example $\varphi_{y_0}(\lambda, t) = t^{h+1} + \lambda (1+t)$, for 
 which the abscissae of vertical tangencies are given by $\lambda+ c \lambda^{1+1/h}$ for some 
 explicit $c$.   
  
 In the general case where $\dot{\mathcal C}$ is reducible, write 
 $\dot\varphi_{y_0}(\lambda, t) \cong \prod_{j=1}^q \xi_j(\lambda, t)$ (up to some invertible element of 
 $\mathcal O(\cd, 0)$), where $\xi_j(0, t) = t^{h_j}+\hot$, and $\sum_j h_j  = h$.  Each branch 
 $\dot{\mathcal C}_j$ is injectively parameterized by $\mu\mapsto (\mu^{h_j}, \theta_j(\mu))$, $\theta_j(0) =0$,
 and substituting this expression in $ \varphi_{y_0}$ as in~\eqref{eq:gamma} 
 gives 
 \begin{equation}\label{eq:gamma_j}
\varphi_{y_0}(\mu^{h_j}, \theta_j(\mu)) = c\theta_j (\mu)^{h+1} +  O({\theta_j (\mu)^{h+2}}) + 
 \mu^{h_j} \psi(\mu^{h_j}, \theta_j(\mu)).
 \end{equation}
 This shows that $m_j   \geq  h_j $, where 
 $m_j = \mult_0(\mathcal {C}, \dot{\mathcal C}_j)$, with equality if and only if $\psi(0, 0)\neq 0$ (note that this condition does not depend on $j$), and 
    $\sigma_j:=m_j/h_j$  
 is the   speed exponent  of the block of $h_j$ 
  tangencies corresponding to $\dot{\mathcal C}_j$. By the additivity of multiplicity, we conclude that 
  $m  = \sum_{j=1}^q h_j \sigma_j$. Combining this relation with  $\sum_j h_j  = h$, shows that 
   {$m=h$ if and only if 
   $\sigma_j = 1$ for every $j$}. As observed above, this property is actually independent of $j$.
 
 The following statement summarizes this discussion. \commentaire{New statement for Exposition Rmk \#3}
 
\begin{propdef}\label{propdef:speed}
Let  $(f_\lambda)$ be as in Normalization~\ref{norm:tang}, with a non-persistent tangency at 0, 
 $\Lambda$ being the unit disk. 
Let $m$ (resp. $h$) is the multiplicity (resp. order) of this tangency. Then there exists   decompositions $m = \sum_{j=1}^q m_j$ and $h = \sum_{j=1}^q h_j$, indexed by the set of  irreducible components of $\set{\fr \varphi/\fr y  = 0}$ at $(0,0)$, such that for every $j$, $m_j\geq h_j\geq 1$. By definition, 
$\sigma_j = m_j/h_j$ is the speed of motion of the block of tangencies corresponding to the 
 $j^{\rm th}$ component, 
and the unfolding has positive speed (we may also say that it is  non-degenerate)
 if $\sigma_j   = 1$ for all (or equivalently some) $j$, 
that is,   $m=h$. 
\end{propdef}

\subsection{Secondary intersections} 
The following result is specific to   complex diffeomorphisms. Note  that for polynomial automorphisms 
of $\C^2$ it also follows from global arguments (see~\cite[\S 9]{bls}). 

\begin{prop}\label{prop:secondary}
If $f:\Omega\to\cd$ is a diffeomorphism with a homoclinic tangency associated to $p$, 
then there are also
 transverse homoclinic intersections between $W^s(p)$ and $W^u(p)$. 
\end{prop}

To get an  intuition of  what is going on, 
let us explain the argument (which is classical) in the oversimplified case 
where the dynamics in linearizable at $p$. In this case we can simply write $f(x,y) = (ux, sy)$. With notation as in \S\ref{subs:conventions} we have a branch $\Delta^u$ of $W^u(p)$ tangent to 
$W_\loc^s(p)$ at $(0, y_0)$, with equation $x = \varphi(y) \cong (y-y_0)^{h+1}$. Pulling back 
this pair of curves by for some 
 iterate $f^k$, we get a branch $\Delta^s  \subset W^s(p)$ tangent 
to $W^u_\loc(p)$ at $(x_0, 0)$, with equation $ y =\psi(x)\cong (x-x_0)^{h+1}$. If $(x,y)\in \Delta^u$ with 
$\abs{x} \ll 1$, then a tangent vector to $\Delta^u$ at $(x,y)$ is   
$v \asymp (x^{h/(h+1)}, 1)$, whose slope is $\asymp x^{ -h/(h+1)}$. Likewise, if 
$(x,y)\in \Delta^s$ with  $\abs{y} \ll 1$, then the  slope of $\Delta^s$ at $(x,y)$ is  
$\asymp y^{ h/(h+1)}$. 
Now let us assume for the moment that some version of the argument principle guarantees that $f^n(\Delta^u)$ intersects $\Delta^s$ for large $n$, so there exists  
$(x_0, y_0)\in \Delta^u$ such that 
$(x_n, y_n) = (u^nx_0, s^ny_0) \in \Delta^s$. Then with notation as above, $v  
\asymp (u^{- nh/(h+1)}, 1)$ so its image under $df^n$ is $\asymp (u^nu^{- nh/(h+1)}, s^n)$ 
 whose slope is $s^nu^{- nh/(h+1)}$. On the other hand the slope of $T\Delta^s$ at 
 $(x_n, y_n)$ is $\asymp s^{n h/(h+1)}  \gg s^nu^{- nh/(h+1)}$ so 
 $f^n\Delta^u$ is transverse to $\Delta^s$. Making this argument rigorous in the non-linearizable case is quite technical, and it is remarkable that in the complex setting 
 all these estimates can be replaced by geometric analysis  considerations.
 
 \begin{proof}
 We keep   notation as above, and use the formalism of 
 horizontal/vertical objects 
  and crossed mappings  from~\cite{HOV2} 
  (or Hénon-like mappings of degree 1 in the language of 
 \cite{henonl}).  
 Consider a thin vertical bidisk around $W^s(p)$ of the form 
 $D(0, \delta)\times \D$. Then for small enough $\delta$, $\Delta^u$ is a horizontal disk of degree $h+1$ in  $D(0, \delta)\times \D$. Likewise, reducing $\delta$ if necessary, 
  $\Delta^s$ is a vertical  disk of degree $h+1$ in  $\D\times D(0, \delta)$. 
 By the Inclination lemma, for large $n$, 
 $f^n$ defines a crossed mapping of degree 1  from  
 $D(0, \delta)\times \D$ to $\D\times D(0, \delta)$. Therefore the graph transform 
 $\el^n\Delta^u$  of $\Delta^u$ (that is, its image under the crossed mapping) is 
 a horizontal disk of degree $h+1$ in $\D\times D(0, \delta)$, so it intersects $\Delta^s$ in 
 $(h+1)^2$ points, counting multiplicities. We claim that for large $n$ all these intersections are transverse. Indeed when $\delta$ is small enough, $\Delta^s \cap (\D\times D(0, \delta))$ is contained 
 in $D(x_0, r)\times D(0, \delta)$, where $r<\abs{x_0}/2$. Now, since  $\el^n\Delta^u$ 
 admits a tangency of order $h$ with 
$W^s_\loc(p)$ and since by the maximum principle it is a topological disk, by the 
Riemann-Hurwitz formula it admits no other vertical tangency (see also Lemma~\ref{lem:riemann-hurwitz}).  So 
$\el^n\Delta^u\cap (D(x_0, r)\times D(0, \delta))$ is the union of $h+1$ horizontal graphs disjoint from 
$D(x_0, r)\times \set{0}$, and close to it. Then by Lemma~\ref{lem:6.4} these graphs must be transverse to $\Delta^s$, and we are done. (Note that we may also apply Lemma~\ref{lem:riemann-hurwitz} since $\Delta^s$ has a horizontal tangency of maximal order.)  
 \end{proof}
 
\begin{rmk}\label{rmk:near_tangencies}
Since there are homoclinic intersections, 
$W^u(p)$ must accumulate itself. In particular there are disks   contained in  $W^u(p)$ arbitrary $C^1$ close to $\Delta^u$. Applying Lemma~\ref{lem:6.4} again then produces transverse 
intersections between these disks and $W^s_\loc(p)$, such that the angle 
between the tangent spaces at the intersection  is arbitrarily small. This observation will be crucial later. 
\end{rmk}

\subsection{Comments on higher dimensional families}  \label{subs:comments_higherdim}
Still working under  the conventions of \S\ref{subs:conventions}, assume in this paragraph   
 that $k:=\dim(\Lambda)>1$. In this context, we say that a non-persistent tangency has positive speed 
 if there exists a smooth 1-dimensional subfamily
 $\Lambda_1\ni\la_0$ along which this property holds.
 Let $\mathcal{T}\subset \Lambda$ be the   locus  where
 $W^s_\loc(p_\la)$ and $\Delta^u_\lambda$ are tangent. \commentaire{We add a formal statement here for Exposition Rmk \#3}

 \begin{lem}\label{lem:higherdim}
 Let $(f_\lambda)_{\lambda\in \Lambda}$ 
 be as in Normalization~\ref{norm:tang}, with a non-persistent tangency. 
 The tangency locus $\mathcal T$ is an analytic hypersurface. If the unfolding has positive speed, 
 $\mathcal T$ is smooth at $\lo$.  
 \end{lem}
  
\begin{proof}
 Arguing exactly as in Lemma~\ref{lem:multiplicity}, we see that 
$\widehat{\P TW^s}$ and  $\widehat{\P T\Delta^u}$ are smooth and of codimension 2 in 
$\Lambda\times \B\times \P^1$, so by~\cite[\S 3.5]{chirka} 
 $\dim (\widehat{\P TW^s}\cap \widehat{\P T\Delta^u})\geq 2(k+1) - (k+3) = k-1$.  
The natural projection 
$\pi_\Lambda:\Lambda\times \B\times \P^1\to \Lambda$ is finite and locally proper in restriction to 
$\widehat{\P TW^s}\cap \widehat{\P T\Delta^u}$, so 
it preserves  dimensions~\cite[\S 3.3]{chirka} and $\mathcal T$ is an analytic subset of dimension 
at least $k-1$. 
Since the tangency is not persistent, $\dim(\mathcal T)<k$, and we conclude that 
$\dim(\mathcal T)= k-1$, that is, 
 {$\mathcal T$ is an analytic  hypersurface}.

If  now we assume that the unfolding has positive speed, and  
if as above $\Lambda_1$ is a 1-dimensional subfamily with non-degenerate unfolding, then
in the 4-dimensional subspace 
 $\Lambda_1\times \B\times \pu$, the 
 intersection $\widehat{\P TW^s}\cap \widehat{\P T\Delta^u}$ 
 (which is reduced to a point) is transverse. 
 Counting dimensions in the tangent space, it is easy to see that 
 it implies the corresponding transversality in  $\Lambda\times \B\times \pu$. And since 
 the fibers of the projections $\pi_{\Lambda}$ and $\pi_{\Lambda_1}$ coincide, we also deduce that 
 in $\Lambda\times \B\times \pu$, $\widehat{\P TW^s}\cap \widehat{\P T\Delta^u}$ is transverse to 
 $\pi_\Lambda\inv(\set{\lo})$. This implies  that 
 $\mathcal T$  is smooth at $\lo$, as asserted.
 \end{proof} 

\section{Unfolding degenerate tangencies}\label{sec:quadratic}

\subsection{Linearization and graph transform estimates}\label{subs:graph_transform}
As in \cite[\S 3]{takens}, a key technical fact in the argument of Theorem~\ref{thm:quadratic} is that in the graph transform, 
  higher derivatives converge faster and  faster to zero. This is obvious for the linear map 
$f(x,y) = (ux, sy)$: indeed the forward iterate of the graph $y  =\gamma(x)$ is $y = \el^n\gamma (x)$, 
with $\el^n\gamma : x\mapsto s^n \gamma( x/u^n)$ so $\smallnorm {(\el^n\gamma )^{(\ell)} }  = O((su^{-\ell})^n)$. 
It is unclear to us whether such an 
 estimate holds in general, and to achieve this, as in \cite{takens}
we use $C^\ell$ linearization, which
 imposes some conditions on the multipliers $u$ and $s$. 
 Even if this belongs to real dynamics, we want to salvage as much complex geometry as possible 
 --in particular we  need to be able to talk about the \emph{complex} multipliers $u$ and $s$, and not only their moduli, which is important in the parameter exclusion in \S\ref{subs:reduction}--
 so the presentation is different from that of Takens (and since this matter is  quite delicate   
 we  actually give more details). 
 In particular we will not switch between different coordinate systems  in the proof of the main theorem, 
 and always stay in holomorphic coordinates, 
 which in our opinion makes the argument neater. 
 
 An important remark is that to achieve 
 $\smallnorm {(\el^n\gamma)^{(j)} }  = O((su^{-j})^n)$ for all $j\leq \ell$ it is not enough to merely know the existence of 
 a system of $C^\ell$ linearizing coordinates: we also need this coordinate system 
  to be flat up to a high  order $K=K(\ell)$ along the separatrices, to 
 prevent lower derivatives of the chart  to spoil the estimate 
 $\smallnorm {(\el^n\gamma)^{(j)} }  = O((su^{-j})^n)$.
 Finally, we need some uniformity of these estimates with respect to parameters. 

\subsubsection{Normal form}  The first stage is to put $f$ in an appropriate normal form, under a non-resonance condition. 
Denote by $\fkm$ the maximal ideal of the local ring of germs of holomorphic functions in $(\cd, 0)$,
 that is, the 
ideal of functions vanishing at the origin.
Recall that if $f(x,y) = (ux, sy)+ \hot$, a \emph{resonance of order $k$} is a relation of the form 
$u = u^a s^b$ or $s = u^as^b$, where $a$ and $b$ are positive integers with $a+b = k$. When 
 $\abs{s}<1<\abs{u}$, this can be rewritten as $u^a s^b = 1$, with $a+b = k-1$. 
 It is well-known that resonances are obstructions to holomorphic (and even formal) linearization. 
 More precisely, a resonance of the form 
 $u = u^a s^b$ (resp.  $s = u^a s^b$) 
 prevents from killing the term $x^ay^b$ in the first (resp. second) 
 component of $f$. Conversely, if $f$ has no resonance up to order $k$,   
 a holomorphic (actually polynomial) change of coordinates brings $f$ to the form 
\begin{equation}\label{eq:form_f}
f(x,y) = (u x+ g_1(x,y), sy+g_2(x,y)), \text{ with } g_1, g_2\in \fkm ^{k+1}.
\end{equation}

The following more precise normal form  is presumably  
 known to some  experts. We include the proof for completeness. 

\begin{prop}\label{prop:sternberg}
Let $f\in \Diff(\cd,0)$ with a saddle fixed point at the origin,
 with eigenvalues $u$ and $s$. If there is no resonance up to order $k+1$, $f$ can be brought to the form 
\begin{equation}\tag{$\star_k$}
f(x,y)  =  (ux(1+ yg_1(x,y)), sy(1+xg_2(x,y))), \text{ with } g_1, g_2\in \fkm^k. 
\end{equation}
by a holomorphic change of coordinates. 
\commentaire{Precision added for \S 4.1.3 (cf. Exposition Rmk \#3). Several comments added in the proof about this. }Furthermore, if $f$ depends holomorphically on a parameter $\lambda$, then this change of coordinate, as well as $g_1$ and $g_2$,  can be chosen to depends holomorphically on $\lambda$ as well. 
\end{prop}

\begin{proof}
We review the proof of the existence of the local stable and unstable manifolds, by checking that the corresponding change of coordinates are sufficiently tangent to identity at the origin, and depend holomorphically on $f$.  Let us first deal with the local unstable manifold. We follow   Sternberg~\cite[Thms 7 to 9, \S 9]{sternberg1}. To stick with the notation of~\cite[p. 823]{sternberg1}, 
we put $T = f\inv$.  Thanks to the non-resonance assumption, we can assume that 
\begin{equation}\label{eq:form_T}
T(x,y) = (u\inv x + \tilde g_1(x,y), s\inv y + \tilde g_2 (x,y)), \text{ with } \tilde g_1, \tilde g_2\in \fkm ^{k+2}.
\end{equation}
\commentaire{added comment on holomorphic variation} If $f$ depends holomorphically on a parameter $\lambda$, 
then so do $\tilde g_1$ and $\tilde g_2$, since they are obtained by explicit transformations on power series. 
We first look for a change of coordinates  $R:(x,y)\mapsto (x, y-\phi(x))$ with $\phi'(0) = 0$
such that $RTR\inv(x,0) = (*, 0)$, so that the axis $\set{y=0}$ is invariant, hence it is 
 the local stable manifold of $T$ (unstable manifold of $f$). The existence (and uniqueness) of such a $\phi$ is guaranteed by the Stable Manifold Theorem. Here we only need to check that $\phi(x) = O(x^{k+2})$.  
With $T$ as in~\eqref{eq:form_T},
  the relation $RTR\inv(x,0) = (*, 0)$ is equivalent to 
 \begin{equation}
 s\inv \phi(x) + \tilde g_2(x,\phi(x)) = \phi(u\inv x + \tilde g_1(x, \phi(x))). 
 \end{equation}
In other words, $\phi$ is a fixed point of the operator $\mathcal D$ defined by 
\begin{equation}
\mathcal D \phi :x\longmapsto   s\lrpar{ \phi(u\inv x + \tilde g_1(x, \phi(x))) -\tilde g_2(x,\phi(x))  }. 
\end{equation}
This is a contracting operator in a suitable Banach space of holomorphic functions on $r\D$ for small $r$, and since $g_1, g_2 \in \fkm^{k+2}$, it preserves the closed subspace of functions vanishing to order 
$k+2$, and the fixed point belongs to this subspace, as desired. \commentaire{added comment on holomorphic variation} Note that
if $f$ depends holomorphically on a parameter $\lambda$, then so do 
  $u$, $s$, $\tilde g_1$ and $\tilde g_2$, hence $\mathcal D$, so the fixed point depends holomorphically on $\lambda$ as well. 

Then we do the same with the stable direction, by a change of coordinate of the form $(x,y)\mapsto (x-\psi(y), y)$, therefore we have shown  that there is a change of coordinate tangent to the identity to the order $k+1$ such that the stable and unstable manifold are the coordinate axes. In these new coordinates, $f$ is of the form
\begin{equation}
f(x,y) = (f_1(x,y), f_2(x,y))  = (ux(1+h_1(x,y)), sy(1+h_2(x,y)), \text{ with } h_1, h_2\in \fkm ^{k+1}. 
\end{equation}

To reach the desired form $(\star_k)$ we linearize $f$ inside $W^{s/u}_\loc(0)$. 
More precisely, we consider the one-dimensional map $h:x\mapsto f_1(x, 0) = ux(1+h_1(x,0)) = ux+ O(x^{k+2})$.  By   Koenigs' theorem it is locally conjugate to $x\mapsto ux$ by some 
 local diffeomorphism $\phi$, which depends holomorphically on $\lambda$ if $f$ does\commentaire{added comment on holomorphic variation}. Moreover, examining the proof (see e.g. \cite{milnor_book}) reveals that 
 $\phi(x)   = x+ O(x^{k+2})$. So if  we conjugate $f$ by $(x,y)\mapsto (\phi(x), y)$, in the new coordinates we get 
 \begin{equation}
f(x,y) =  (ux(1+\tilde h_1(x,y)), sy(1+\tilde h_2(x,y)), \text{ with } \tilde h_1,\tilde  h_2\in \fkm ^{k+1}, 
\end{equation}
 and furthermore $ \tilde h_1(x, 0) = 0$, hence $\tilde  h_1 (x,y) = y \hat h_1(x,y)$, with 
 $\hat h_1(x,y)\in \fkm^k$. 
 Repeating this operation in the stable direction (which does not affect the form 
 of the first coordinate of $f$) concludes the proof.
 \end{proof}

\subsubsection{$C^\ell$ estimates.}
 Let $f$ be a diffeomorphism defined in   $2 \B$, 
with a saddle fixed point at the origin, which is 
of the form  $(\star_k)$ of Proposition~\ref{prop:sternberg}. 
Fix a constant $\rho>0$ such that 
\begin{equation}\label{eq:rho}
1+\rho \leq \abs{u}\leq 1+\rho\inv \text{ and } 1+\rho \leq \abs{s}\inv\leq 1+\rho\inv 
\end{equation}
If $\norm{g_1}_{2  \B}$ and $\norm{g_2}_{2 \B}$ in $(\star_k)$ are small enough, then by the Cauchy estimates, the 
graph transform $\el$ acting on horizontal graphs in $\B$ is well defined. We leave it as an exercise to the reader to check that the condition  
\begin{equation}
\label{eq:maxg}
\max(\norm{g_1}_{2  \B},\norm{g_1}_{2  \B})\leq \frac{\rho}{10} 
\end{equation}
 is sufficient. 
 Note that if $f$ is of the form $(\star_k)$ in some small neighborhood of $0$, 
 then by scaling the coordinates 
we can assume that it is defined in $2\B$ and achieve any desired bound 
on $\max(\norm{g_1}_{2  \B},\norm{g_1}_{2  \B})$.

\begin{prop}\label{prop:graph_transform}
For every   integer $\ell\geq 2$, there exists  $k = k(\ell, \rho)$ such that if 
 $f$ is a  diffeomorphism defined in   $2 \B$ of the form $(\star_k)$   satisfying 
 \eqref{eq:rho} and \eqref{eq:maxg}, then
there exists     a constant $C = C(\ell,\rho)$\commentaire{$\norm{g_1}_{2  \B},\norm{g_1}_{2  \B}$ removed, cf minor comment \#3}
such that for every $j \leq \ell$, $\norm{(\el^n\gamma)^{(j)}}_\D \leq C\abs{s u^{-j}}^n$. 
\end{prop}

\begin{proof} The proof proceeds by constructing a $C^r$-diffeomorphism linearizing $f$, which is 
tangent to the identity up to order $r$ along the axes, for a sufficiently large  $r$.  In a first stage we estimate the required value of $r$, and then we  follow Sternberg~\cite{sternberg2} to show that if $k$ is large enough, such a  linearization exists.  

 \medskip

\noindent{\emph{Step 1. Estimation of the order of differentiability $r$. }}

Here we show that there exists $r = r(\ell, \rho)$ such that 
if there exists  a $C^r$-diffeomorphism $R$ whose image contains a neighborhood 
of $\overline \D\times \set{0}$,   
linearizing $f$, i.e. 
  $R\circ L  = f\circ R$ with $L(x,y) = (ux,sy)$ (for convenience the 
  notation here is as in~\cite{sternberg2}, except that $f$ is denoted by $T$ there), and 
   which  is tangent to the identity up to order $r$ along 
  $\set{y=0}$, then $\norm{(\el^n\gamma)^{(j)}}_\D \leq C\abs{s u^{-j}}^n$ for every $j\leq \ell$. 
  
 Write 
 $R = (R_1, R_2)$ and denote by $\pi_1$, $\pi_2$ the  coordinate projections, so that $R_1 - \pi_1$ and $R_2 - \pi_2$ have vanishing first $r$ derivatives along $\set{y=0}$.  
 Start with $R\inv(\set{y = \gamma(x)})$ which is of the form $y = \psi(x)$ 
 and iterate $L$ to get 
 a sequence of graphs  $y = \psi_n(x) =  s^n\psi(xu^{-n})$, whose image under $R$ 
 is   $\set{y  = \gamma_n(x)}$, where $\gamma_n = \el^n\gamma$. 
  Unwinding the definitions gives 
$R_2(x, \psi_n(x)) = \gamma_n(R_1(x, \psi_n(x)))$,
 and we have to estimate the derivatives of $\gamma_n$. 
A caveat is in order here: $x$ is a complex variable and $\gamma_n$ is holomorphic,
but  $R$ is not, so formally we  have to  write $x = x_1+ix_2$ and deal with the partial derivatives
$\fr_{x_1}^{i_1}\fr_{x_2}^{i_2}$. 
For notational ease, we not dwell on this point and do as if everything was holomorphic 
(this does not change the structure of the estimates). 

Write $\widetilde \gamma_n = \gamma_n( R_1(x, \psi_n(x)))$ and let us estimate the derivatives of 
$\widetilde \gamma_n  - \psi_n$. The $j^{\rm th}$-derivative of 
$\widetilde \gamma_n(x) - \psi_n=  (R_2 - \pi_2) (x , \psi_n(x))$   is a sum of terms involving 
the partial derivatives of $R_2-\pi_2$ multiplied by
 polynomial expressions  in the derivatives of $\psi_n$ 
(which can be 
computed exactly using  the Faa Di Bruno formula). Analyzing this expression and using 
 $\fr^j(R_2 - \pi_2)(x,y) = O( {y}^{r-j})$ and $\psi_n^{(j)} = O((s u^{-j})^n)$, it is not 
difficult to convince oneself that the dominant term is the one obtained 
 by differentiating with respect to the 
first variable (i.e. the first two real variables) of $(R_2 - \pi_2)$, that is, the only term which comes with no additional
multiplicative factor. This term 
  is of order of magnitude $O(\abs{\psi_n}^{r-j}) = O(\abs{s}^{n(r-j)})$, and we conclude that 
$\big\vert{\widetilde \gamma_n^{(j)} - \psi_n^{(j)}}\big\vert \lesssim \abs{s}^{n(r-j)}$. 

We now choose the order $r$ such that for every $j\leq \ell$, ${s}^{n(r-j)} = o((su^{-j})^n)$. For this it is enough that 
$\abs{s}^{r- j}<\abs{su^{-j}}$, that is $\abs{s}^r < \abs{s (su\inv)^j}$, hence  it suffices that 
$r> (1+\alpha)\ell+1$, where $\alpha$ is such that $\abs{u}\inv = \abs{s}^\alpha$, and  the choice 
 \begin{equation}\label{eq:choice_r}
 r(\ell, \rho) = 2+ \lrpar{   1+\frac{\ln(1+\rho\inv)} {\ln(1+\rho)} } \ell
 \end{equation}
works. 

 Let us now estimate $\big\vert{ \gamma_n^{(j)} }\big\vert$. Recall that  $\widetilde \gamma_n = \gamma_n( R_1(x, \psi_n(x)))$, and that 
 at this stage we know that $\big\vert{\widetilde \gamma_n^{(j)}  }\big\vert  = O ((su^{-j})^n)$. Write $h_n (x) = R_1(x, \psi_n(x))$. This is 
 a diffeomorphism such that $h_n(x) - x =O(s^{rn})$. Arguing as above shows that 
 $  \fr h_n(x) = 1+ O(s^{n(r-1)})$ and for $2\leq j\leq r$,  $\fr^j h_n(x)  = O(s^{n(r-j)})$. It follows that the inverse diffeomorphism 
 $h_n\inv$ satisfies the same estimates. Indeed, the expression for the $j^{\rm th}$   derivative of $h_n\inv$  is a rational function whose denominator is a power of  $  \fr h_n( y)$, $y = h_n\inv(x)$,
 and whose numerator is a polynomial expression in 
 the $\fr^ih_n  (y)$, $i\leq j$, with a single term of order $j$. Plugging in the estimate $\fr^i h_n(x)  = O(s^{n(r-i)})$
 shows that $\fr^j h_n\inv(x)  = O(s^{n(r-j)})$. 
 
 Now we write $ \gamma_n = \widetilde \gamma_n \circ h_n\inv$. 
 Taking the $j^{\rm th}$ derivative of this expression gives a sum of terms of the form 
 $P_i(\fr h_n\inv, \ldots , \fr^j h_n\inv) \fr^i \widetilde \gamma_n(h_n\inv)$, with 
 $i\leq j$. Furthermore, when $i=j$, $P_j$ depends only on $\fr h_n\inv$, so the corresponding 
  term is of order of magnitude $O(\fr^j \widetilde \gamma_n) = O((su^{-j})^n)$. In 
 the remaining terms with $i<j$, all monomials of $P_i$ 
 involve higher derivatives of 
 $h_n\inv$, so they are bounded by  $O(s^{n(r-j)})$. Our choice of $r$ guarantees that $s^{n(r-j)} = o ((su^{-j})^n)$ so we conclude 
 that $\big\vert{ \gamma_n^{(j)} }\big\vert \leq C \abs{su^{-j}}^n$, as announced.

\medskip

\noindent{\emph{Step 2.  Construction of a $C^r$-linearization.}}

We   follow step by step  the proof of   Theorem~1 pp. 628-629 in~\cite{sternberg2} 
to show that there exists $k = k(r, \rho)$ 
such that if $f$ is of the form   $(\star_k)$ and satisfies  \eqref{eq:rho} and \eqref{eq:maxg}, then 
a $C^r$-diffeomorphism $R$  linearizing $f$ as in Step 1 exists. 
The assumption that $f$ is of the form 
  $(\star_k)$ is precisely the conclusion of Lemma~7 in~\cite{sternberg2}.  
   Sternberg constructs $R$    by  first   prescribing it   on some fundamental domain of $\B \setminus \set{xy = 0}$ for the action of $L$, then 
extending $R$ to   $\B \setminus \set{xy = 0}$ by the  equivariance, and finally showing that 
 $R$ together with its first $r$ derivatives approach  the identity along the axes.  (We enlarge a little bit $\B$   so that $R(\B)\supset \overline\D\times \set{0}$.)
Bounding  these derivatives relies on the   iteration  
of an  operator $D_T$ introduced in equation (19) p.~629 (all references in the next few lines are relative to \cite{sternberg2}). The order  $k$  (which is denoted by $q$ there)
is chosen at this stage, 
according to   the requirement that the estimate~(20) p.~169 holds with
$\alpha<\abs{u}^{r}$.
Thus $k$  depends only on $r$ and $\rho$, hence ultimately on $\ell$ and $\rho$. 
Then, the growth of $\norm{D_T^n }$  is governed 
by   $\alpha   = \alpha(\ell, \rho)$, $\norm{g_1}_{2\B}$ and $\norm{g_2}_{2\B}$ (see 
equations~(20) and (21) on  p.~629; $g_1$, $g_2$ correspond  to the error term $F$), and we are done. 
\end{proof}

\begin{rmk}\label{rmk:finv}
In the proof of Theorem~\ref{thm:quadratic} we will actually use this result with the stable and unstable directions reversed; of course for this it is enough to apply the result to $f\inv$. 
We prefer to stick with this presentation here to make the comparison with~\cite{sternberg2} easier.
\end{rmk}

\begin{rmk}\label{rmk:uniformity_k}
It is a much studied and difficult problem to study the dependence of the integer 
$k$ in terms of $(r, \rho)$.  No explicit estimates are  given in~\cite{sternberg2}; 
results in this direction can be found e.g.  in \cite{belitsky, sell}. 
\end{rmk}

\subsubsection{Comments on families}\label{subs:remark_families}
If  $f$ belongs to some family $(f_\lambda)$, the property of 
having no resonance up to a certain order is   open   in  
 parameter space. In this case    
 Proposition~\ref{prop:sternberg} shows that in this open set, 
 $f_\lambda$ is reduced to the form $(\star_k)$ 
 in a fixed neighborhood of the origin,
 by  a change of coordinates 
 depending   holomorphically on $\lambda$. 
 After proper rescaling we may assume  that $f_\lambda$ is defined in $2\B$
 and the corresponding 
 $\norm{g_{1, \lambda}}_{2\B}$ and $\norm{g_{2, \lambda}}_{2\B}$  are locally uniformly bounded, 
 and the quantity $\rho$ from Equations~\eqref{eq:rho} and~\eqref{eq:maxg} 
 can be chosen locally uniformly. 
It follows from the statement of  Proposition~\ref{prop:graph_transform} that 
the implied constant 
$C$  is locally uniformly bounded as well. \commentaire{No verification is now left to the reader in this paragraph (cf. Exposition Rmk \#3)}
 
\subsection{Dynamical slope}\label{subs:slope}
In \cite[\S 3.4]{takens}, Takens  defines a notion of \emph{angle of crossing}, that we prefer to call \emph{dynamical slope}, whose purpose is to make the estimates in the Inclination Lemma more precise, and plays an important role in the main argument. To understand the idea, let us   consider the real 
 linear case: 
if $f(x,y)  = (ux, sy)$ and $v$ is a tangent vector at $(0, y_0)$, $y_0\neq 0$, 
 with slope $m$, then its $n^{\rm th}$
  image is a tangent vector at $(0, y_n) = (0, s^ny_0) $ with slope $m_n=(su\inv)^nm$.  
So if $\alpha>1$ is such that $su\inv = s^\alpha$, we see that $ {m_n}{y_n^{-\alpha}}$ 
is an invariant quantity, which by definition is the dynamical slope. In particular, if the dynamical slope is non-zero, then $m_n\cong (su\inv)^n$.
This definition can be extended to the non-linear case by linearization,
 under an appropriate non-resonance 
assumption. 

In the complex setting, we cannot directly extend this definition (because $y_n^{-\alpha}$ does not make sense), 
so we content ourselves with explaining  that there is a well-defined notion of having \emph{non-zero} 
dynamical slope, which suffices  for our needs (see Remark~\ref{rmk:slope2} for   further discussion).   
 
For the notion of slope to make 
 sense, we need the coordinates to be properly normalized, so we will 
assume that $f\in \Diff(\C^2, 0)$ is of the form $(\star_k)$ (for some $k\geq 1$ to be determined below)
and consider the action 
of $f$ (denoted by $f_\varstar$)
on projective tangent vectors $[v]\in \P T\cd\rest{W^s_\loc(0)}$. If $[v] = [v_1:v_2]$ is not tangent 
to $W^s_\loc(0)  = \set{x=0}$ its slope is by definition $\slope([v]) = v_2/v_1$. 

\begin{prop}\label{prop:slope}
Assume that  $f\in \Diff(\C^2, 0)$ is  of the form $(\star_k)$, 
where $k$ is such that $\abs{s}^k< \abs{u}\inv$. 
Then there exists an invariant  holomorphic section $Z$  of the bundle 
  $\P T\C^2\rest{W^s_\loc(0)}\to W^s_\loc(0)$, disjoint from $[TW^s]$,
 such that if 
$[v]\in \P T\C^2\rest{W^s_\loc(0)}$  neither  belongs to the image of $Z$ nor to $[TW^s]$, then 
$\slope (f_\varstar^n [v]) \cong (u\inv  s)^n$. 
\end{prop}
  
By definition $Z$ is the \emph{zero dynamical slope section}. Since $Z$ is uniformly transverse to 
$W^s_\loc$, it follows that any tangent vector which is sufficiently close to $TW^s_\loc$, but not tangent to it, has non-zero dynamical slope. 

\begin{rmk}\label{rmk:slope}
With the formalism of Equation~\eqref{eq:rho}, if there is no resonance up to order 
\begin{equation}\label{eq:k'}
k'(\rho) = 2 + \frac{\ln(1+\rho\inv)}{\ln(1+\rho)}, \end{equation} then 
Proposition~\ref{prop:sternberg} guarantees that 
$f$ can be put in the form $(\star_k)$,  with  $\abs{s}^k< \abs{u}\inv$, and 
Proposition~\ref{prop:slope} applies in the corresponding  coordinates. 
\end{rmk}

\begin{proof} \commentaire{Proof largely rewritten, cf. Exposition remark \#5.}
We work in coordinates where $f$ is of the form $(\star_k)$ and 
  study the dynamics of $f_\varstar$ on $\P T\C^2\rest{W^s_\loc(0)}$, which is a $\P^1$-bundle over the disk. The ``central fiber'' $\P T_0\cd$ is globally attracting and contains two fixed points, one saddle 
$(0, [E^s])$ corresponding to the stable direction and one attracting $(0, [E^u])$
corresponding to the unstable direction. Fix local coordinates $(y,m)$ near $(0, [E^u])$, where 
$m$ is the slope. Since the differential of $f$ along $W^s_\loc(0) = \set{x=0}$ is of the form 
$$df_{(0,y)}  = 
\begin{pmatrix}
u(1+yg_1(0,y)) & 0 \\
syg_2(0,y) & s(1+yg_2(0,y))
\end{pmatrix} $$ 
we infer that 
$f_\varstar$ 
 is expressed as  $f_\varstar (y,m) = (sy, (  s u\inv) m +h(y,m))$, with $h\in \mathfrak{M}^{k+1}$. 
 
Thanks to our assumption on $k$, the Poincaré-Dulac theorem guarantees that 
$f_\varstar$ can be linearized at the origin. The precise form 
of the linearizing map is important in the argument: fortunately, due to the special form of $f$, 
the linearization is easily obtained as follows.  Let 
$L(y,m) = (sy,(su\inv)m)$ and write $L^{-n}f_\varstar^n (y,m) = (y, \varphi_n(y,m))$. Put also 
$f_\varstar^n (y,m) = (y_n, m_n)$
Then 
\begin{equation}
\varphi_{n+1}(y,m)= (su\inv)^{-(n+1)}  m_{n+1} 
= (su\inv)^{-(n+1)} \lrpar{(su\inv)m_n + h(y_n, m_n)}
\end{equation} hence
\begin{equation}\label{eq:PD}
\varphi_{n+1}(y,m)  - \varphi_n(y,m) = (su\inv)^{-(n+1)} h(y_n, m_n).
\end{equation}
Since $\abs{su\inv}<\abs{s}$, $\norm{(y_n, m_n)} = O(\abs{s}^n)$, so 
$\norm{h(y_n, m_n)} = O(\abs{s}^{(k+1)n})$, and since  
$\abs{s}^{k+1}< \abs{su\inv}$,  we deduce from~\eqref{eq:PD} that 
$\abs{\varphi_{n+1} - \varphi_n}$ converges exponentially to 0. Therefore 
$L^{-n}f_\varstar^n$ converges to the linearizing map  
\begin{equation}\label{eq:PD2}
\Phi \colon (y,m)\longmapsto 
  \lrpar{y, m+ (su\inv)\inv h(y,m) + \sum_{n=1}^\infty (su\inv)^{-(n+1)} h(y_n, m_n)} .
\end{equation}
 
In the linearizing coordinates,  the dynamics becomes $L:(y', m')\mapsto (sy', (su\inv) m')$, so if 
$m'\neq 0$, the second coordinate  of $L^n(y',m')$ decays like $c(su\inv)^n$ with $c\neq 0$. 
Back to the initial coordinates, since by~\eqref{eq:PD2} the linearizing map 
$\Phi$ is of the form 
$\Phi(y,m) = (y,m+\varphi(y,m))$  with 
$\varphi\in \mathfrak{M}^{k+1}$ and $\abs{s}^{k+1}< \abs{su\inv}$,
this estimate is not spoiled by terms coming from $\varphi$, and  
we conclude  that there is an 
invariant  holomorphic curve $Z= W^{ws}(0, [E^u])$ transverse to the central fiber (associated to the ``slow'' eigenvalue $s$)  such that if 
$(y,m)\notin Z$ then $m_n\cong (su\inv) ^n$. Thus we have defined the announced 
section $Z$  in some neighborhood of $(0, [E^u])$, and we extend it to $\P T\C^2\rest{W^s_\loc(0)}$ 
by pulling back. Finally, to finish the proof, it is enough to observe that 
any $v\in \P T\C^2\rest{W^s_\loc(0)}$ not 
tangent to $\set{x=0}$ is eventually attracted by the unstable direction, so the previous analysis applies. 
\end{proof}
  
\begin{rmk} \commentaire{New remark, cf Exposition remark \#5}
In the dissipative case we have $\abs{s}<\abs{u}\inv$, and $f$ can be put in   $(\star_1)$ form so Proposition~\ref{prop:slope} applies without further assumption. 
\end{rmk}
  
\begin{rmk}\label{rmk:slope_family}\commentaire{New remark for Exposition remark \#3}
If $(f_{\la})_{\la\in \Lambda}$ is a holomorphic family of diffeomorphisms 
of the form $(\star_k)$, with $\abs{s_\lambda}^k<\abs{u_\lambda}\inv$, then 
 it is clear from the proof that the zero slope section $Z$ depends holomorphically on $\lambda$. 
\end{rmk}

\begin{rmk}\label{rmk:slope2}
Even if it is not clear how to define the dynamical slope as a number in the complex case, 
the notion of two tangent directions  
having the same dynamical slope does make sense. Indeed, with notation as in the previous proof, we can pull back the foliation $\set{m' = \cst}$ by the linearizing coordinates, which   defines sections of constant dynamical slope in $\P T\C^2\rest{W^s_\loc(0)}$. Still, it is not obvious to decide whether two 
tangent directions at different points of $W^s_\loc(0)$ have the same slope or not. Since the idea of taking distinct slopes plays an important role in~\cite{takens}, we have to modify the concluding
  argument  in  the proof of Proposition~\ref{prop:main} so that considering one  non-zero slope is sufficient.   
\end{rmk}


\subsection{Reduction of Theorem~\ref{thm:quadratic} to  Proposition~\ref{prop:main}} \label{subs:reduction} ~ \commentaire{Subsection rewritten completely and many details added for Exposition Remark \#2}
Suppose that $(f_\lambda)_{\lambda\in \Lambda}$ is a family
admitting a non-persistent homoclinic tangency at $\lambda_0=0$, of order $h_0$,
associated to a saddle fixed point $p$ 
that is not persistently resonant, as in Theorem~\ref{thm:quadratic}. It is enough to treat the case where 
$\Lambda\simeq \D$. 

Since by Proposition~\ref{prop:secondary} 
there are also  transverse homoclinic intersections at $\lambda=0$, 
 by unfolding the tangency  as described in the introduction,
we obtain infinitely  many secondary tangency parameters  
(still associated to  $W^{s/u}(p_\lambda)$). 

Our purpose in this paragraph is to show that in any neighborhood of $\lambda_0$ 
we can construct a    special
  homoclinic  tangency parameter $\lambda^\varstar$ 
and a topological disk $\Lambda^\varstar\subset \Lambda$ containing $\lambda^\varstar$ 
such that after putting $f_{\lambda^\varstar}$ in Normalization~\ref{norm:tang}, 
the following properties (P1-5) hold. 
Reducing $\Lambda$ if necessary, 
we fix  a constant $\rho>0$   such that the estimate \eqref{eq:rho} holds for every 
$\lambda\in \Lambda$.  Let $k =\max (k(h_0+1, \rho), k'(\rho))$, 
where these numbers are respectively defined in 
 Proposition~\ref{prop:graph_transform} and Remark~\ref{rmk:slope}.

\begin{itemize}
\item [(P1)] \textit{for every $\lambda\in \Lambda^\varstar$, $f_\la$ is of the form $(\star_k)$ and satisfies \eqref{eq:maxg}.}
\item [(P2)]  \textit{There is a horseshoe $\mathcal H_\lambda$, moving holomorphically with $\lambda\in \Lambda^\varstar$, whose local stable manifolds in $\B$ are vertical graphs 
accumulating $W^s_\loc(p_\lambda) = \set{x=0}$ and having non-zero dynamical slope at their intersection with  $W^u_\loc(p_\lambda)$.}
\end{itemize}

The  components of  $W^s(p)\cap \B$ contained in the semi-local stable manifolds of 
 $\mathcal H_\lambda$  form a countable dense subset of the local stable lamination 
 $W^s_\loc(\mathcal H_\lambda)$, which we enumerate as $(W^{s,i}_\lambda)_{i\in \N}$ 
 (each depending holomorphically on $\lambda$). 
%
%
 Let $\mathcal T\subset\Lambda^\varstar$ be the   locus where a tangency occurs 
between $\Delta^u_\lambda$ and a branch
 $W^{s,i}_\lambda$.
\begin{itemize}
 \item[(P3)] \textit{For every $\lambda \in \mathcal T$ the   order of tangency between
  $\Delta^u_\lambda$ and the corresponding  branch
 $W^{s,i}_\lambda$
 is equal to a constant $h$ and there is a unique point of tangency.}
 \item[(P4)] \textit{The parameter of tangency between $\Delta^u_\lambda$ and  
 $W^{s,i}_\lambda$ is unique, and the multiplicity of tangency is equal to a constant $m$. }
\end{itemize}
 
 The proof will proceed by   successively choosing 
    tangency parameters $\lambda_i$ and corresponding (shrinking) neighborhoods
   $\Lambda_i$ to secure these properties one after the other. Each 
  time we will consider a suitable iterate of the tangency point to bring it back to 
  $W^s_\loc(p_\lambda)$ and  renormalize the coordinates. 

Start with $\lambda_0 = 0$, fix a disk $\Lambda_0$ containing $\lambda_0$
 and assume that Normalization~\ref{norm:tang} 
holds for the initial tangency. 

By Proposition~\ref{prop:secondary},   there is a 
transverse   intersection between $W^s(p_0)$ and $W^u(p_0)$ at some point $\eta_0$
which by Smale's procedure allows to  creates a horseshoe $\mathcal H_{0,0}$ 
invariant by $f_0^N$ for some $N$. 
Taking some (negative) iterate, the intersection point can be chosen to belong to 
a small neighborhood of 0 in 
 $W^u_\loc(0)$, and increasing $N$,  
the horseshoe can be chosen to be arbitrarily close to 
$\set{0,  \eta_0}$. In particular the 
  the semi-local stable manifolds of the
 horseshoe form a Cantor set of  vertical graphs in  
 $\B$,  close to $\set{x=0}$  and accumulating it. Reduce $\Lambda_0$ if needed so that 
$\mathcal H_{0,0}$  can be followed holomorphically as $\mathcal H_{0, \lambda}$ in $\Lambda_0$. 

Reducing  $\Lambda_0$ again  if necessary, 
  we may assume that   
$\Delta^u_\la$ as a horizontal submanifold of degree $h_0+1$ for every 
$\lambda\in \Lambda_0$.
and also that for $ \la\in \Lambda_0\setminus\set{\lambda_0}$, $\Delta^u_\la$ is transverse to $\set{x=0}$. 
Hence,  by  the  compactness of $\fr\Lambda_0$, 
 there is a  uniform $r$ such that  for $\lambda\in \fr\Lambda_0$,
 $\Delta^u_\lambda\cap D(0,r)\times \D$ is a union 
of $h_0+1$ graphs over $D(0,r)$, and by the inclination lemma, there is a uniform 
$M$ such that $f^M_\lambda (\Delta^u_\lambda)\cap \B$ is a union of graphs. Therefore, 
replacing $\Delta^u_\lambda$ by $f^M_\lambda (\Delta^u_\lambda)\cap \B$, which does not affect 
Normalization~\ref{norm:tang}, 
we may assume that 
 the vertical tangencies escape $\B$ in the sense that for $\lambda$ close to 
 $\fr\Lambda_0$, $\Delta^u_\la$ is a union of $h_0+1$ horizontal graphs. 
 
 As above, denote by $(W^{s,i}_\lambda)_{i\in \N}$
the  components of  $W^s(p)\cap \B$ contained in the semi-local stable manifolds of 
 $\mathcal H_{0,\lambda}$, and note that since the   tangencies escape $\B$ for $\la\in \fr\Lambda_0$, no tangency
 between $\Delta^u_\lambda$ and some $W^{s,i}_\lambda$ can be persistent. Denote by 
 $\mathcal T_0$ the union of all these tangency parameters.

 \begin{lem}\label{lem:perfect}
 $\mathcal T_0$ is a countable perfect set. 
 \end{lem}

\begin{proof}
Denote by $\mathcal T_{0,i}\subset \Lambda$   the  set of parameters for which  a  
tangency occurs between $\Delta^u_\la$ and  $W^s_{i, \lambda}$ so that 
$\mathcal T_0 = \bigcup_i \mathcal T_{0,i}$. By Lemma \ref{lem:creating_tangencies}, 
each $\mathcal T_{0,i}$ is a non-empty finite set, so $\mathcal T_0$ is countable. 
 Since $\bigcup W^{s,i}$ is dense in a Cantor set of graphs, it is perfect. Thus, given $i_0\in \N$,   there exists a sequence $(i_j)$ such that $W^{s,{i_j}}$ converges to $W^{s,{i_0}}$, hence $\widehat{\P T W}^{s,{i_j}}$ converges to $\widehat{\P T W}^{s,{i_0}}$ (notation as in \S\ref{subs:lifting}). The persistence of proper intersection shows that   $\widehat{\P T W}^{s,{i_0}}\cap \widehat{\P T \Delta^u}$ is accumulated 
 by  $\widehat{\P T W}^{s,{i_j}}\cap \widehat{\P T \Delta^u}$, and since the projection to
  $\Lambda$ is finite we infer that 
 any point of $\mathcal T_{0,i_0}$ is accumulated by the $\mathcal T_{0,i_j}$, 
 so $\mathcal T_0$ is perfect, as announced. 
\end{proof}

With $k$ chosen as above, since there is no 
persistent resonance between the multipliers of $p_\lambda$, there is a locally 
finite subset $F_k$ of $\Lambda$ such that outside $F_k$ there is no 
resonance up to order $k+1$.  Since $\mathcal T_0$ is perfect, 
 $\mathcal T_0\setminus F_k$     is relatively open and dense in $\mathcal T_0$.
Pick  $\lambda_1\in \mathcal T_0\setminus F_k$  and a   disk 
$\Lambda_1 $  containing    $\lambda_1$,  contained in $\Lambda_0$ and disjoint from $F_k$.  By Proposition~\ref{prop:sternberg}, for $\lambda \in \Lambda_1$
we can choose local coordinates near $p_\lambda$, 
depending holomorphically on $\lambda$,  
in which $f_\lambda$ is of the form  $(\star_k)$. Taking a smaller 
neighborhood and rescaling it, we may assume that $f_\lambda$ is defined in $2\B$ and 
\eqref{eq:maxg} holds there. Replacing the  tangency point and 
  $\Delta^u_\lambda$ by a sufficiently large iterate (here and in the following for convenience 
  we still denote this iterate  by $\Delta^u_\lambda$)  we may assume that Normalization~\ref{norm:tang} holds at $\lambda_1$ in the new coordinates, and that    the vertical tangencies of $\Delta^u_\lambda$ 
  escape $\B$  along $\fr\Lambda_1$. At this stage, Property (P1) is satisfied. The bidisk $\B$ and corresponding coordinates will not be changed anymore.

For $\lambda\in \Lambda_1$ the notion of dynamical slope is 
well defined, and by Remark~\ref{rmk:near_tangencies} 
there is a transverse intersection  $\xi_{\lambda_1}$ between 
$W^s(p_{\la_1})$ and $W^u_\loc(p_{\la_1})$ 
whose dynamical slope is non-zero (see the comments after 
Proposition~\ref{prop:slope}). As before we choose $\xi_{\lambda_1}$ sufficiently close to 0 
so that its semi-local stable manifold is  a vertical graph.  
Given this intersection, we can repeat
 the horseshoe construction from the above paragraph, to get a new horseshoe $\mathcal H$ (this will be the final one)
 with the additional property 
that the stable manifolds of the horseshoe have non-zero dynamical slope. 
Indeed, if the horseshoe is sufficiently thin (which implies that the iterate $N$ is large), 
all the semi-local stable manifolds in the ``half  piece'' of $\mathcal H$ close to $\xi_{\lambda_1}$  
 intersect $W^u_\loc(p_{\lambda_1})$ with a non-zero dynamical slope. Then 
 by invariance we infer that   the same holds for all semi-local stable manifolds of the horseshoe, 
 except for $W^s_\loc(p_{\la_1})$.

Since the property of having non-zero slope is open (see Remark~\ref{rmk:slope_family}), 
we can find     an open 
neighborhood $\Lambda_2\subset \Lambda_1\setminus F_k$\commentaire{cf. Minor comment \#4}
of $\lambda_1$ 
where $\mathcal H$ can be followed holomorphically and the slope remains non-zero. Take some further iterate of $\Delta^u_\lambda$ so that   $\Delta^u_\lambda$ is a union of disks for $\lambda\in \fr\Lambda_2$. 
 Properties (P1) and  (P2) are now satisfied in $\Lambda_2$.
   For convenience replace $f_\la$ by $f_\la^N$ so that 
  $\mathcal H_{\lambda}$ is fixed by $f_\lambda$.  
  
We let  $\mathcal T_2\subset \Lambda_2$ be the new corresponding 
  tangency locus. Note that since the vertical tangencies of 
  $\Delta^u_\lambda$ escape $\B$ 
at $\fr\Lambda_2$, all these tangencies are non-persistent,  therefore, 
 exactly as for $\mathcal T_0$,
   Lemma~\ref{lem:perfect} shows that  $\mathcal T_2$ is a countable perfect set. 
For $\lambda\in \mathcal T_2$ there are  a number of tangencies between $\Delta^u_\lambda$ and
$W^s(p)$. Pick $\lambda_3\in \mathcal T_2$ together with    
a tangency point $t(\lambda_3) \in \Delta^u_{\lambda_2}\cap W^s_{\lambda_2}(p_{\lambda_2})$
whose order is minimal among all tangencies appearing in $\mathcal T_2$. Denote this order 
by $h$. 

Take some iterate of $t(\lambda_3)$ (and of $\Delta^u_{\lambda_3}$) so that Normalization~\ref{norm:tang} holds back again, and reduce $\Lambda_2$ to some $\Lambda_3$ so that the new
$\Delta^u_\lambda$ is a horizontal disk of degree $h$ for $\lambda\in \Lambda_3$, 
and let $\mathcal T_3 = \mathcal T_2\cap \Lambda_3$. 
 For $\lambda\in \mathcal T_3$, $\Delta^u_\lambda$ is a horizontal manifold of degree 
 $h+1$ in  $\B$, with a tangency with a semi-local stable manifold that  is a
  vertical graph, 
 thus by Lemma~\ref{lem:riemann-hurwitz}
this tangency is necessarily of order 
 $\leq h$. Since $h$ was chosen to be minimal, the order of tangency 
 is equal to $h$, 
 and again by  Lemma~\ref{lem:riemann-hurwitz}, the tangency point is unique, so (P3) holds. 
Observe that by uniqueness and the fact that $\mathcal T_3$ is perfect, 
the tangency point  moves continuously 
 with $\lambda\in \mathcal T_3$. 
 
 Now we minimize the multiplicity. As before, enumerate as $(W^s_i)$ the vertical 
 components of $W^s(p)$ contained in $\mathcal H$. For every   $i$, 
 $\widehat{\P T\Delta^u}\cap \widehat{\P T W^s_i}$, if non-empty,  consists of  one or several points 
 (which must then correspond to  distinct parameters since for fixed
  $\lambda\in\Lambda_3$ the tangency point is unique), 
 with an associated multiplicity. Pick a parameter $\lambda_4\in \mathcal T_3$
 and an intersection point where this multiplicity is minimal and denote it by $m$.   
 Then by upper-semicontinuity of the  multiplicity,
  there is an open disk $\Lambda_4\subset \Lambda_3$ around $\lambda_4$ in which 
 all tangencies  have minimal multiplicity $m(\lambda)\equiv m$ (and order $h$). 
 Furthemore, restricting $\Lambda_4$ again if necessary, the tangency 
 parameter is unique\commentaire{cf. Minor comment \#5}. Indeed, for $\lambda$ close to $\lambda_4$, the intersection point of $\widehat{\P T\Delta^u}$ and  $\widehat{\P T W^s_i}$, which has 
  minimal multiplicity in $\Lambda\times \B\times \P^1$, cannot split into several points of smaller multiplicity. Thus, property (P4) holds. 
 
The reduction is now complete, and we   put $\lambda^\varstar = \lambda_4$. 
 For the last time, we iterate the tangency point and the component $\Delta^u_\lambda$, 
  and  reduce $\Lambda_4$ if necessary to a smaller disk $\Lambda^\varstar$ 
 so that Normalization~\ref{norm:tang} holds at $\lambda^\varstar$, 
and  the new $\Delta^u_\lambda$ is a horizontal disk in $\B$ of degree $h$. Put $\mathcal T = \Lambda^\varstar\cap \mathcal T_3$. 
 
 Theorem~\ref{thm:quadratic} then reduces to the following proposition:
  
\begin{prop}\label{prop:main}
With notation as above, $h = m=1$.   
\end{prop}
 
%
%

 Before embarking to the proof we make a last observation: after the restrictions that we made, it is not necessarily 
true that   every component   $W^s_{i, \lambda}$ is tangent to  $\Delta^u_\lambda$ for some parameter (this was guaranteed for  $\Lambda_2$ thanks to the property that vertical tangencies escape $\B$ 
at $\fr\Lambda_2$, but not afterwards). However this holds for every 
 component $W^{s,i}$ in some neighborhood of $\set{x=0}$. In the following 
 we fix $r>0$ with the property  that every 
 branch $W^{s,i}$ contained in $D(0,r)\times \D$ admits a tangency with $\Delta^u_\lambda$ for some $\lambda\in \Lambda^\varstar$. We  furthermore request that  for any vertical 
 graph contained in $D(0,r)\times \D$,  the  graph transform  $f\inv(\Gamma)\cap \B$ is contained in 
 $D(0,r)\times \D$ as well.

 For convenience in the proof of Proposition \ref{prop:main} we put $\lambda^\varstar = 0$ 
and rename $\Lambda^\varstar$ into $\Lambda$.

\subsection{Proof of Proposition \ref{prop:main},   part I: $h=1$} \label{subs:proof_part1}
We argue by contradiction so assume that $h\geq 2$. 
Fix a vertical component of 
$  W^s_\loc(\mathcal H_\lambda)\cap W^s(p_\lambda)$ contained in $D(0,r)\times \D$, and denote it by 
$\Gamma^s_\lambda$.
 Let $\Gamma^s_\lambda\cap \set{y=0} =  (\alpha(\lambda), 0)$, 
 so that $\alpha$ is a non-vanishing holomorphic function.
 Let $\Gamma^s_{\lambda, n}$  be the truncated pull-back
of $\Gamma^s_\lambda$ by $f_\lambda^n$, and note that 
$\Gamma^s_{\lambda, n}\cap \set{y=0} =  (u_\lambda^{-n} \alpha(\lambda), 0)$. Thanks to 
 property (P1), the estimate of Proposition~\ref{prop:graph_transform} holds uniformly in $\Lambda$ (see   
the comments in \S~\ref{subs:remark_families}).

For $\lambda = 0$ there is a tangency of order $h$ between the branch $\Delta^u_0$ of 
$W^u(p_\lambda)$ and $W^s_\loc(p_\lambda)$ at $(0, y_0)$, 
which unfolds with the parameter $\lambda$. 
This   initial     tangency may split into several vertical tangencies as 
$\lambda$ evolves, but by (P4) for every $n$
 there is a unique   parameter $\lambda_n$ at which
  $\Gamma^s_{\lambda, n}$ and $\Delta^u_\lambda$ are tangent, 
and the  tangency point is unique by (P3), and of order $h$ by assumption. 
Denote  by $y(\lambda_n)$ the $y$-coordinate of the tangency point. 
Note that $\lambda_n\to 0$ as $n\to\infty$. 

The multipliers satisfy 
 $u_\lambda = u_0(1 + O(\lambda))$ and $s_\lambda = s_0(1+O(\lambda))$.  We will repeatedly use the following elementary
 observation: if $\mu_n$ converges exponentially fast to zero, then 
 $u_{\mu_n}^n\sim u_0^n$ (see~\eqref{eq:ulambda} for the argument). 

\begin{lem}\label{lem:lambda_n}
There is a unique speed exponent $\sigma$, in particular $\sigma = m/h$. The tangency parameter $\lambda_n$ satisfies 
$\abs{\lambda_n} \cong \abs{u_0}^{-n/\sigma}$. More precisely we have 
$\lambda_n^m \cong u_0^{-nh}$. 

\end{lem}

 
\begin{proof}
By the case $\ell=1$ of 
Proposition~\ref{prop:graph_transform},  there is a constant $C$ such that for every $\lambda$, 
$\Gamma^s_{\lambda, n} $ is contained in the ``tube'' 
$D(u_\lambda^{-n} \alpha(\lambda), C\abs{u_\lambda^{-1} s_\lambda }^n)\times \D$.  
For $\lambda = \lambda_n$, 
$\Delta^u_{\lambda_n}$ is a  horizontal submanifold 
of degree $h+1$ in $\B$, with a tangency of order $h$ with the vertical graph 
$\Gamma^s_{\lambda_n, n}$. By the maximum principle, every component 
of $\Delta^u_{\lambda_n}\cap 
D(u_\lambda^{-n} \alpha(\lambda), C\abs{u_\lambda^{-1} s_\lambda }^n)\times \D$ 
is a holomorphic disk. Lemma~\ref{lem:vertical_tangencies} then implies that 
 $\Delta^u_{\lambda_n}$ admits $h$ vertical tangencies in 
 $D(u_\lambda^{-n} \alpha(\lambda), C\abs{u_\lambda^{-1} s_\lambda }^n)\times \D$.

By Proposition-Definition~\ref{propdef:speed}, the vertical tangencies of $\Delta^u_{\lambda}$ are decomposed in blocks of $h_j$ tangencies 
moving like $\lambda^{\sigma_j}$, $\sigma_j\geq 1$. Denote by $x_{i, j}(\lambda)$ the abscissae of 
vertical tangencies, where  $1\leq j\leq q$ and  $ 1\leq i\leq h_j$. 
These do not necessarily define holomorphic functions, but we know that 
$\abs{ x_{i, j}(\lambda)}\cong \abs{\lambda}^{\sigma_j}$ as $\lambda\to 0$. 
By the first part of the proof, for all $i,j$, $x_{j, i}(\lambda_n)$
belongs to $D(u_\lambda^{-n} \alpha(\lambda), C\abs{u_\lambda^{-1} s_\lambda }^n)$. 
Taking moduli we see 
that $\abs{x_{j, i}(\lambda_n)}\sim \abs{u_{\lambda_n}}^{-n}  \abs{\alpha(\lambda_n)}$.
From the two previous relations we get that 
$\abs{\lambda_n}^{\sigma_j}\cong \abs{u_{\lambda_n}}^{-n}  \abs{\alpha(\lambda_n)}$. Since $\lambda_n\to 0$, this shows that $\lambda_n$ decays exponentially, so 
$\abs{u_{\lambda_n}}^{-n}\sim \abs{u_0}^{-n}$, hence 
$\abs{\lambda_n}^{\sigma_j}\cong  \abs{u_0}^{-n}$, from which it follows that 
$\sigma_j$ is independent of $j$ and $\abs{\lambda_n} \cong  \abs{u_0}^{-n/\sigma}$.
 From the discussion in \S\ref{subs:speed} (in particular Equation~\eqref{eq:gamma_puiseux} and the discussion following it) we have that $x_{j, i}(\lambda)^h\cong \lambda^m$, so the same reasoning 
 shows that $\lambda_n^m\cong u_0^{-nh}$, as desired. 
\end{proof}

\begin{rmk}\label{rmk:similarity} \commentaire{New remark related to Major remark \#3 and minor comment \#11}
Record for future reference that what we only 
need from Proposition~\ref{prop:graph_transform} in this lemma is the existence of a bound of the form 
$o(u_\lambda^{-n})$ for the slope of $\Gamma^s_{\lambda,n}$. We will see in the proof of Theorem~\ref{thm:persistent} below that such a bound always holds. 
We also used the uniqueness of the tangency parameter. 
Thus, in a general setting, 
for tangencies with $h=m=1$, the proof applies without modification to show  that 
$\lambda_n\cong u_0^{-n}$. 
\end{rmk}

Recall that the equation of $\Delta^u_\lambda$ near $(0, y_0)$ is of the form $x=\varphi_\lambda(y)$, with $\varphi_\lambda(y)  = c(y-y_0)^{h+1}+ O((y-y_0)^{h+2})+ \lambda \varphi_1(\lambda, y)$. Let 
$\mathcal C^{(h)}$ be the curve in $(\lambda, y)$ space defined near $(0, y_0)$ 
by $ \varphi_\lambda^{(h)}(y)   = 0$ (derivative with respect to the $y$ variable). Since 
\begin{equation}
 \varphi_\lambda^{(h)}(y)   \cong   (y-y_0) + O((y-y_0)^2 + O(\lambda),
 \end{equation}
  we see that $\mathcal C^{(h)}$ is smooth and  locally a graph over the $\lambda$-coordinate, of the form 
$y  = \psi(\lambda)$, with $\psi(0) = y_0$.  \commentaire{We do not define anymore the integer $q$ here, cf. Major Remark \#1}

Let $x= \gamma_{\lambda_n,n}(y)$ be the equation of $\Gamma^s_{\lambda_n, n}$, and
note that by Proposition~\ref{prop:graph_transform} (for $\ell=h$), we have 
$\abs{\gamma_{\lambda_n,n}^{(h)} (y(\lambda_n))}\lesssim  
\abs{u_{\lambda_n} \inv s_{\lambda_n}^{h}} ^n$. 
Since  $\Delta^u_{\lambda_n}$ and $\Gamma^s_{\lambda_n, n}$ are tangent to order $h$ 
at $y(\lambda_n)$ we obtain  
\begin{equation}
\abs{\varphi_{\lambda_n}^{(h)} (y(\lambda_n)) }\leq 
C \abs{u_{\lambda_n} \inv s_{\lambda_n}^{h}} ^n
\end{equation}
and since $(\lambda_n)$ decays exponentially we get 
  $\big( u_{\lambda_n}\inv  s_{\lambda_n}^{h}\big) ^n \sim \big(u_0 \inv  s_0^{h} \big)^n$, thus 
\begin{equation}\label{eq:ylambda_n}
\abs{\varphi_{\lambda_n}^{(h)} (y(\lambda_n)) }\leq 
C \abs{u_0 \inv s_0^{h}} ^n. 
\end{equation}
Since there exists $\delta>0$ such that  in the neighborhood of $(0,y_0)$, 
$\big\vert{\varphi_{\lambda}^{(h+1)}}\big\vert\geq \delta$, from~\eqref{eq:ylambda_n} we deduce 
 that there exists $\widetilde y_n$ with 
 $\abs{y(\lambda_n) - \widetilde y_n} \lesssim \abs{u_0 \inv s_0^{h}} ^n$   such that 
 $\varphi_{\lambda_n}^{(h)} (\widetilde y_n) =0$. Hence 
  $(\lambda_n , \widetilde y_n)\in \mathcal C^{(h)}$, 
  so $\widetilde y_n = \psi(\lambda_n)$,
 and we record the estimate
 \begin{equation}\label{eq:varphi_prime_lambda1}
 \abs{y(\lambda_n) - \psi(\lambda_n)} \lesssim \abs{u_0 \inv s_0^{h}} ^n
 \end{equation}
 
 Using again the fact that $\Delta^u_{\lambda_n}$ and 
 $\Gamma^s_{\lambda_n, n}$ are tangent  
at $y(\lambda_n)$, we get $\gamma_{\lambda_n,n}' (y(\lambda_n)) = \varphi_{\lambda_n}'(y(\lambda_n))$.
By Proposition~\ref{prop:graph_transform} for $\ell=2$ we have 
\begin{equation}\label{eq:gamma_n}
\abs{\gamma_{\lambda_n,n}' (y(\lambda_n))- \gamma_{\lambda_n,n}' (0)}\lesssim \norm{\varphi''_{\lambda_n}} \lesssim \abs{u_0\inv s_0^2}^n.
\end{equation}
Now recall that $\Gamma_\lambda^s$ was chosen so that its dynamical slope is non-zero for every $\lambda\in \Lambda$. This 
 is an invariant property so it holds for $\Gamma_{\lambda_n,n}^s$, which by Proposition~\ref{prop:slope}
 implies that 
 $\gamma_{\lambda_n,n}' (0)\cong (u_0\inv  s_0)^n$. 
 By~\eqref{eq:gamma_n} we get $\gamma_{\lambda_n,n}' (y(\lambda_n))\cong (u_0\inv s_0)^n$, hence 
 $\varphi_{\lambda_n}'(y(\lambda_n)) \cong (u_0\inv s_0)^n$. 
 Since  $\abs{\varphi''_{\lambda_n}}$ is uniformly bounded, 
 from~\eqref{eq:varphi_prime_lambda1} we get 
 \begin{equation}\label{eq:varphi_prime_lambda2}
 \varphi_{\lambda_n}' (\psi(\lambda_n))  =  \varphi_{\lambda_n}' (y(\lambda_n)) + O \lrpar{(u_0 \inv s_0^{h})^n}   \cong (u_0\inv  s_0)^n
 \end{equation}
 (note that   $h\geq 2$ is used exactly here).
 In particular $\lambda \mapsto \varphi_\lambda'(\psi(\lambda))$ is not identically 0, and since it is holomorphic and $\varphi_\lambda'(y_0) = 0$\commentaire{cf. Minor comment \#6}, there exists an integer $q\geq 1$ such 
 that $\varphi_\lambda'(\psi(\lambda))\cong \lambda^q$ as $\lambda\to 0$\commentaire{Major Remark \#1 (continued): the integer $q$ is introduced here}. Thus~\eqref{eq:varphi_prime_lambda2} rewrites as  
$\lambda_n^q\cong (u_0\inv  s_0)^n$ and using 
Lemma~\ref{lem:lambda_n} we finally get that as $n\to\infty$,
\begin{equation}\label{eq:su}
u_0^{-nhq}\cong (u_0\inv  s_0)^{nm}. 
\end{equation}

\begin{lem}
If $a$ and $b$ are non-zero complex numbers such that $a^n\cong b^n$ as $n\to\infty$,  then $a=b$. 
\end{lem}
\begin{proof}
Indeed if $z\in \C$ is such that $z^n\to c\neq 0$ then $z=1$ and $c=1$.  
\end{proof}

Thus the relation~\eqref{eq:su} implies that $u_0^{-hq} = (u_0\inv  s_0)^{m}$, that is, 
$u_0^{m-hq} = s_0^m$. 
 Now  we repeat  this entire reasoning for every parameter 
$\lambda\in \mathcal T$ sufficiently close to 0, and it follows that 
 for every such $\lambda$ there is a relation 
of the form $u_\lambda^{-hq } = (u_\lambda\inv  s_\lambda)^{m}$, where $q$ depends a priori on $\lambda$.
Since  $\abs{u_\lambda}$ and $\abs{s_\lambda}$ are  uniformly bounded 
 away from 0 and 1, and $h$ and $m$ are fixed, 
 we infer  that $q$ 
is uniformly bounded. Therefore we can 
 select an infinite subset  $\mathcal T'\subset \mathcal T$  where 
 the relation $u_\lambda^{m-hq} = s_\lambda^m$ holds for a fixed $q$, so by analyticity
$u_\lambda^{m-hq} = s_\lambda^m$ for every $\lambda\in \Lambda$. This contradicts the non-existence of persistent resonances, thereby completing    the proof. \qed

\subsection{Proof of Proposition \ref{prop:main},   part II: $m=1$} This is a rather direct consequence 
of the uniqueness of the tangency parameter and the argument principle, 
so the result is  simpler in the complex case  than in the real case.

At this stage we know that $h=1$, so $\Delta^u_\lambda$  admits 
 a unique vertical tangency. Denote by $x(\lambda)$ its first coordinate.
With notation as in \S\ref{subs:speed}  since $h=1$, we have    
 $x(\lambda) = \varphi_{y_0}(\lambda, \theta(\lambda)) = c\lambda^m+\hot$ (see~\eqref{eq:sigma}), 
 and from  the relation 
 $\sum h_j = h$, we see that  $\dot{\mathcal C}$ is irreducible.  
 For notational simplicity we rescale the parameter space so that 
 $c=1$. 
 Assume by way of contradiction that $m\geq 2$.
 As before, fix a  
  vertical component   $\Gamma^s_{\lambda}$ of 
$  W^s_\loc(\mathcal H_\lambda)\cap W^s(p_\lambda)$ contained in $D(0,r)\times \D$,
  with $\Gamma^s_\lambda\cap \set{y = 0}= (\alpha(\lambda), 0)$, with $\alpha(\lambda)\neq 0$, 
  and let $\Gamma^s_{\lambda, n}$ its  $n^{\rm th}$ truncated pull back.
  We may assume that $\alpha(\lambda)$ does not move too much in the sense that 
  $\abs{\alpha(\lambda) - \alpha(0)}\leq \abs{\alpha(0)}/10$. 
   We will 
  show that for large $n$ there are $m$ distinct parameters such that 
 $\Delta^u_\lambda$ is tangent to $\Gamma^s_{\lambda, n}$, 
 which contradicts  property (P4).
 For this, we use the following facts: (i) for large $n$,
 $\Gamma^s_{\lambda, n}$ is very close to the vertical line through 
 $(\alpha(\lambda)u_\lambda^{-n}, 0)$, and (ii) the equation 
 $x(\lambda) = \alpha(\lambda)u_\lambda^{-n}$ has $m$ solutions. 
 
The exact formulation of step (ii) is the following:
 
\begin{lem}\label{lem:elementary_analysis}
If  $ \delta < \abs{u_0}\inv$   is fixed, then  for 
  sufficiently large $n$,  there are $m$ disjoint topological disks $\Lambda_{n, i}\subset \Lambda$, 
  $1\leq i\leq m$, with $\Lambda_{n, i} \subset D(0, C \abs{u_0}^{-n/m})$, 
  in which 
$\lambda \mapsto x(\lambda ) - \alpha(\lambda) u_\lambda^{-n} $ realizes a 
biholomorphism
$\Lambda_{n, i}\to D(0, \delta^n)$. 
\end{lem}

\begin{proof}[Proof of Lemma \ref{lem:elementary_analysis}] We  rely on the following   fact 
 from elementary  complex analysis, which we leave as an exercise to the reader: if $f$ is a holomorphic function on $\D$ such that $f(0)  = 0$, $f'(0) = 1$ and $\abs{f'}\leq M$, then there exists $\rho= \rho(M)$ 
and a domain $\Omega$ with $D(0, \rho/2)\subset \Omega\subset D(0, 3\rho/2)$ such that $f\rest{\omega}:\Omega\to D(0, \rho)$ is a biholomorphism. By rescaling it holds with $f$ holomorphic in $D(0, R)$, 
$\abs{f'(0)} = d$, $\abs{f'}\leq dM$ and the image radius is $\rho Rd$.

Recall that by assumption $x(\lambda) = \lambda^m+\hot$
 Let us first show that there are $m$ solutions to the equation $x(\lambda)   = \alpha(\lambda) u_\lambda^{-n}$ close to the origin.
 Write  $x(\lambda) = g(\lambda)^m$, where $g$ is holomorphic in some disk 
$D(0, R)$ and $g'(0) = 1$, so that  by the above result 
$g$ is a univalent  map  $\Omega \to D(0,  \rho R)$, for some domain $\Omega$.

Recall that if 
 $\lambda$ is exponentially small (i.e. $\abs{\lambda}\leq (1-\eta)^n$ for some $\eta>0$), then
$u_\lambda^{-n}\sim  u_0^{-n}$: indeed $u_\lambda = u(1+O(\lambda))$, so 
\begin{align}\label{eq:ulambda}
u_\lambda^{-n} &= u_0^{-n}(1+ O((1-\eta)^n))^n = u_0^{-n} \exp(-n\ln(1+O((1-\eta)^n) ))   \\ & \notag
= u_0^{-n} \exp(n O((1-\eta)^n) )  
  \sim u_0^{-n}. 
\end{align}
 It follows that for every choice $\zeta_i$, $1\leq i \leq m$ 
  of  $m^{\rm th}$ root of $\alpha(\lambda) u_\lambda^{-n}$, we have a solution $\lambda_{n,i}$ 
of $g(\lambda)   = \zeta_i$ in $\Omega$ with $\lambda_{n,i}\asymp \abs{u_0}^{-n/m}$. 
At $\lambda_{n,i}$ we have $x'(\lambda_{n,i}) 
\sim m \abs{\lambda_{n,i}}^{m-1} \asymp
\abs{u_0}^{-n({m-1})/{m}}$, while reasoning as in \eqref{eq:ulambda} we get that 
$\frac{d}{d\lambda}\big\vert_{\lambda = \lambda_{n, i}}( \alpha(\lambda) u_\lambda^{-n}) 
= O(\abs{u_0}^{-n})$. Let $\beta$ be such that 
$ \delta\abs{u_0}^{(m-1)/m}< \beta < u^{-1/m}$. Then in $D(\lambda_{n,i}, \beta^n)$ we get that 
$\abs{\frac{d}{d\lambda}(x(\lambda) - \alpha(\lambda) u_\lambda^{-n}} \leq C \abs{u_0}^{-n({m-1})/{m}}$, 
so by the preliminary fact, 
$\lambda\mapsto x(\lambda) - \alpha(\lambda) u_\lambda^{-n}$ realizes  a biholomorphism 
 from an approximately round domain $\Lambda_n$
 of size $\asymp \beta^n$  about $\lambda_{n,i}$
 to a disk centered at the origin and 
 of radius $\asymp  \beta^n\abs{u_0}^{- n(m-1)/m}\gg \delta^n$, and we are done. 
\end{proof}

Fix $\delta$ such that 
$\max(\abs{s_0u_0\inv}, \abs{u_0}^{-1-1/m})<\delta < \abs{u_0}\inv$, and let $(\Lambda_{n,i})_{1\leq i\leq m}$ 
be as in the previous lemma.  
Let us conclude the proof of Proposition \ref{prop:main} 
by showing that for every $i$ 
there is a parameter  in $\Lambda_{n, i}$ for which 
$\Delta^u_\lambda$ is tangent to $\Gamma^s_{\lambda, n}$. 
Using the fact that 
$\Lambda_{n, i} \subset D(0, C\abs{u_0}^{-n/m})$ and arguing as in~\eqref{eq:ulambda}, for  $\lambda\in \Lambda_{n, i}$ we get 
\begin{equation}\label{eq:alpha}
\abs{\alpha(\lambda)u_\lambda^{-n}  - \alpha(0)u_0^{-n} }= \abs{u_0}^{-n} \abs{\alpha(\lambda)  \exp(n\ln(1+O(\lambda))) - \alpha(0)} \lesssim  \abs{u_0}^{-n} \cdot n\abs{u_0}^{-n/m}
\end{equation}
Fix $1\leq i\leq m$ and 
let as before $\lambda_{n,i}$ be the unique solution of 
$x(\lambda_{n,i} ) - \alpha(\lambda_{n,i}) u_{\lambda_{n,i}}^{-n}=0$  in $\Lambda_{n,i}$. Let us check that the assumptions of Lemma~\ref{lem:creating_tangencies} are satisfied in  the bidisk 
$D(x(\lambda_{n,i}), \delta^n/2)\times \D$, for the parameter space $\Lambda_{n,i}$. First, for $\lambda = \lambda_{n,i}$, $\Delta^u_{\lambda_{n,i}}$ admits a vertical tangency at $x=  x(\lambda_{n,i})$ so it is not a union of graphs. By~\eqref{eq:alpha}, for $\lambda\in  \overline{\Lambda_{n, i}}$ we have 
\begin{equation}
\abs{\alpha(\lambda) u_\lambda^{-n} - \alpha(\lambda_{n,i}) u_{\lambda_{n,i}}^{-n}}\lesssim n \abs{u_0}^{-n\lrpar{1+\unsur{m}} } = o(\delta^n),
\end{equation}
hence  for $\lambda\in \fr \Lambda_{n, i}$ and large enough $n$
\begin{equation}
\abs{x(\lambda) - x(\lambda_{n,i})}  = \delta^n - \abs{\alpha(\lambda) u_\lambda^{-n} - \alpha(\lambda_{n,i}) u_{\lambda_{n,i}}^{-n}} \geq \frac{\delta^n}{2}.
\end{equation}
Thus  for $\lambda\in \fr \Lambda_{n, i}$ the unique vertical tangency of $\Delta^u_{\lambda}$ lies outside 
$D(x(\lambda_{n,i}), \delta^n/2)\times \D$, so in this bidisk 
the horizontal manifold $\Delta^u_{\lambda}$ 
is a union of (two) holomorphic graphs. Now by Proposition \ref{prop:graph_transform}, 
  for $\lambda \in \Lambda_{n,i}$, 
$\Gamma^s_{\lambda, n}$ is a vertical graph in $\B$, through  $(\alpha(\lambda) u_\lambda^{-n}, 0)$, with slope bounded by $C\abs{s_\lambda u_\lambda\inv}^n \lesssim \abs{s_0u_0\inv}^n$. 
Therefore, by \eqref{eq:alpha}, its first coordinate its 
contained in  
\begin{equation}
D\lrpar{\alpha(\lambda)u_\lambda^{-n}, C \abs{s_0u_0\inv}^n} \subset 
D\lrpar{\alpha(\lambda_{n,i}) u_{\lambda_{n,i}}^{-n}, C \abs{s_0u_0\inv}^n+ C n \abs{u_0}^{-n\lrpar{1+\unsur{m}}}}
\end{equation}
which is contained in $D(\alpha(\lambda_{n,i}) u_{\lambda_{n,i}}^{-n} , \delta^n/4)= D(x(\lambda_{n,i}), \delta^n/4)$ for large enough $n$. Therefore 
Lemma~\ref{lem:creating_tangencies}  applies and produces a tangency between 
$\Delta^u_{\lambda}$ and 
$\Gamma^s_{\lambda, n}$ in each $\Lambda_{n,i}$, which is the desired contradiction. \qed

 \subsection{The case of complex Hénon maps: proof of Corollary~\ref{cor:quadratic_henon}}
 The proof of Corollary~\ref{cor:quadratic_henon} relies on the weak stability theory 
 of~\cite{tangencies}: see the Appendix for a brief review. 
 

\begin{proof}[Proof of Corollary~\ref{cor:quadratic_henon}] 
At the parameter $\lo$ there is a non-persistent 
tangency between $W^s(p_0)$ and $W^u(p_0)$, where $p_0$ is 
some saddle periodic point (whose continuation is   denoted by $p_\la$). 
If there is no persistent resonance between the multipliers of $p_\la$, we are done by Theorem~\ref{thm:quadratic}. Otherwise, let 
$\mathcal S_0$ be the set of saddle points at $\lo$, which is dense in 
$J^\varstar_0$ by \cite{bls}. Every $q_0\in \mathcal S_0$ admits a local continuation $q_\lambda$.

The first claim is that for every $q_0$ in $\mathcal S_0$, there is a parameter $\la_1$ arbitrarily close to $\lo$ for  which there is a homoclinic tangency associated to $q_{\lambda_1}$. 
Indeed at $\lambda_0$, by \cite[Thm 9.6]{bls}, 
$p_0$ and $q_0$ belong to the same homoclinic class, and this property persists in a small neighborhood $\Lambda'$ of $\lo$. As in \S\ref{subs:conventions}, fix a disk 
$\Delta^u_0 \subset W^u(p_0)$ tangent to the local stable manifold of $p_0$ at $t(\lo)$. By the inclination lemma, we can find two  sequences of disks $\Gamma_n^s(q_0)\subset W^s(q_0)$ and 
$\Gamma^u_n(q_0)\subset (W^u(q_0))$, which are 
respectively vertical and horizontal submanifolds in $\B$
converging in the $C^1$ sense 
to  $W^s_\loc (p_0)$ and $\Delta^u(p_0)$. This whole picture persists  under small pertubations, 
so it makes sense to talk about $\widehat {\P T\Gamma^u_n}(q)$  and 
$\widehat {\P T\Gamma^s_n}(q)$
 as submanifolds of 
$\Lambda'\times \B\times \P^1$, and the $C^1$ convergence in $\B$ of $\Gamma^u_n(q_\la)$ and 
$\Gamma^s_n(q_\la)$ at every parameter implies that 
$\widehat {\P T\Gamma^u_n}(q) \to \widehat {\P T\Delta^u}(p)$ and 
$\widehat {\P T\Gamma^s_n}(q) \to \widehat {\P TW^s_\loc}(p)$ in the Hausdorff sense. Therefore 
the  persistence 
of proper intersections entails
 that $\widehat {\P T\Gamma^u_n}(q)$ and $\widehat {\P T\Gamma^s_n}(q)$ 
 must intersect for large $n$, which is the desired result. 

We then conclude the proof by observing that there  must exist $q_0\in \mathcal S_0$ such that 
there is no persistent resonance between the eigenvalues of $q_\la$. Indeed, since there is a non-persistent tangency, the family $(f_\la)$ is not  weakly stable 
 in any neighborhood of $\lo$ (see the Appendix for more explanations). 
 So some saddle point $q\in \mathcal S$ must change type, and since we are in a 
dissipative setting, it  bifurcates to a sink. On the other hand 
assume that  there is a persistent relation of the form $u_\la^as_\la^b =1$ for the eigenvalues of $q_\la$.
Note that $j_\la = \abs{u_\la s_\la}<1$ in the family, so using $\abs{u_0}> 1 > \abs{s_0}$, we get 
 $a>b>0$. Then $\abs{u_\la}^a \abs{s_\la}^b = \abs{u_\la}^{a-b} j_\la^b =1$ and we get that 
 $\abs{u_\la}>1$ at all parameters, contradicting the fact that $q_\la$ becomes   
 a sink in some domain of 
   parameter space. 
\end{proof}

\section{Bifurcations from persistent tangencies}\label{sec:persistent}

\subsection{Proof of Theorem~\ref{thm:persistent}} We argue by contraposition, so assume that
$(f_\lambda)_{\lambda\in \Lambda}$ is a   weakly stable substantial 
 family of polynomial automorphisms of $\C^2$ with a persistent tangency. 
 Recall that substantial means that either the family is dissipative or there is no persistent 
 resonance between eigenvalues of periodic points (see the Appendix). It is a necessary  condition 
 for  the weak stability theory of~\cite{tangencies} to work.
 
Without loss of generality  we may assume  that $p$ is fixed. Our purpose is to show that 
 $\lambda\mapsto \frac{\ln \abs{u_\lambda}}{\ln \abs{s_\lambda}}$ is constant by adapting the theory of moduli of stability from~\cite{palis_moduli, newhouse-palis-takens}. \commentaire{Incorrect sentence removed here, cf. Minor comment \#7.}
By analytic continuation 
 it is enough to prove the constancy of this function 
 on  an open set of parameters, 
 so   we may freely replace $\Lambda$ by some open subset during the proof. 
 
Working locally in $\Lambda$, we may choose local coordinates depending holomorphically on 
$\lambda$ and iterate the tangency point so that Normalization~\ref{norm:tang} holds, with $\abs{y_0}<1/2$. 
  The order of tangency $h$ is upper-semi-continuous for the analytic Zariski topology (see Remark~\ref{rmk:semi-continuity-persistent}), so we 
   further reduce to some  open 
 subset of $\Lambda$   (still denoted by $\Lambda$)
  where $h$ is minimal. In this  case, arguing as in the proof of property (P3) in 
 \S~\ref{subs:reduction},  
 the tangency point  $\tau_\lambda$ is unique and 
 can be followed holomorphically. 
 
 \commentaire{paragraph added for Exposition Comment \#4} The proof proceeds in two main steps. 
 In Step 1, which is similar to~\cite{palis_moduli}, we work with a fixed parameter and identify 
 a certain region in a branch $\Gamma_n^s$ of $W^s(p)$ close to the tangency
 such that for every $r_n$ in this region, $f^{-m_n}(r_n)$ converges to a compact subset of $W^s(p)\setminus\set{p}$ for a well-controlled time sequence $m_n\approx 
 \frac{\ln\abs{u}}{ \ln\abs{s\inv}} n$. 
 We also check (Step 1') the uniformity of certain quantities appearing in this analysis.  
  In Step 2, we show that for our weakly 
 $\jstar$-stable family, thanks to the properties of holomorphic motions, 
 this picture can be followed continuously with the parameter $\lambda$, which 
 in turn forces $\frac{\ln\abs{u_\lambda}}{ \ln\abs{s_\lambda}}$ to be constant. 
 
 \medskip 
 
 \noindent{\bf Step 1}.   Since we work with a fixed parameter,   for notational simplicity we drop the 
 mention to $\lambda$. 
 As in \S\ref{subs:proof_part1}
 we fix a vertical graph $\Gamma^s$ in $\B$, contained in $W^s(p)$, whose equation is 
 $x = \gamma(y)$, with $\gamma(0) = \alpha$.     Its cut-off pull-back under $f^n$ is 
 $ \Gamma^s_n = \set{x= \gamma_n(y)}$, with $\gamma_n(0) = \alpha u^{-n}$, 
 $\norm{\gamma_n} = O(\abs{u}^{-n})$ and $\norm{\gamma_n'} = o(\abs{u}^{-n})$. 
 For the last estimate we may observe that the 
 saddle fixed point is automatically $C^1$ linearizable 
 (see Hartman~\cite[Thm IV, p. 235]{hartman})
 and argue as in \S\ref{subs:graph_transform}. Alternatively, we may simply use the existence of a $f\inv$-invariant cone field 
 $\mathcal C^s$  with $\mathcal C^s_x = \set{(v_1, v_2)\in T_x\C^2, \ \abs{v_2}\geq C \abs{v_1}}$ as follows: choose   $u'<\abs{u}$ and $s'> \abs{s}$ such that  $ (u')\inv s'  < \abs{u\inv}$.   
If $\B$ is small enough, and  $x\in \B$ and $f\inv(x)\in\B$, then for 
 $v = (v_1, v_2)\in T_x\C^2$,  then  $d f\inv_x(v) = (w_1, w_2)\in \mathcal C_{f\inv(x)}$, with 
 $\abs{w_1}\leq (u')\inv \abs{v_1}$ and $\abs{w_2}\geq (s')\inv\abs{v_2}$. So
 we see that the multiplicative factor 
 for the slope $\abs{w_1/w_2}$  is not greater than  $ (u')\inv s'  < \abs{u\inv} $ and we are done.

 Now we intersect $\Gamma^s_n$ and the branch $\Delta^u$ 
 of $W^u(p)$ tangent to $W^s_\loc(p)$ at $\tau =(0, y_0)$, whose equation is 
 $x= \varphi(y) = \varphi(y_0+t) = c t^{h+1} + O(t^{h+2})$, with $c\neq 0$. 
 From the previous slope estimate,
  $\gamma_n(y_0+ t)  = \alpha u^{-n} + o(u^{-n})$, where the $o(\cdot)$ is uniform in $t$,   the solutions of 
 $\gamma_n(y) = \varphi(y)$ are of the form 
 \begin{equation}\label{eq:polygon}
 \tilde y_n^{(i)} = y_0+\tilde t_n^{(i)}, \text{ where }
\tilde t_n^{(i)} =  \lrpar{\frac{\alpha}{c} u^{-n}}^{1/(h+1)} + o\lrpar{\abs{u}^{- n/(h+1)}}
, \ 1\leq i\leq h+1
 \end{equation}
corresponding to  the various choices of $(h+1)^{\rm th}$ roots (this is elementary : argue as in Lemma~\ref{lem:elementary_analysis}).
 
\begin{lem}\label{lem:distance}
Fix  $0<\e<\frac{1}{10}\abs{\alpha/c}^{1/(h+1)}$, and 
  $r_n\in \Gamma_n^s$   of the form 
\begin{equation}\label{eq:distance}
 r_n = (\gamma_n(y_0+t_0), y_0+t_0) \text{, for a given }
\abs{t_0} \leq \e \abs{u}^{- n/(h+1)}.
\end{equation} Then $d(r_n, \Delta^u)\asymp \abs{u}^{-n}$, and the implied constants depend only on $\e$.
\end{lem}
 
 
\begin{proof}
Switching to the sup norm for notational simplicity, we have to minimize 
\begin{align}
\notag d((\gamma_n(y_0+t_0), y_0+t_0),  &(\varphi(y_0+t), y_0+t)) = 
\max\lrpar{ \abs{\gamma_n(y_0+t_0) - \varphi(y_0+t)}, \abs{t- t_0}} \\ & = 
 \max\lrpar{ \abs{\alpha u^{-n} +o(\abs{u}^{-n})  - c t^{h+1} + O(t^{h+2})  }, \abs{t- t_0}} 
\end{align}
for small $t$. First considering  $t=t_0$ we see that 
$d(r_n, \Delta^u)\lesssim  \abs{u}^{-n}$ so the minimal distance is achieved in the domain 
$\abs{t-t_0}\lesssim \abs{u}^{- n}$ (otherwise the second coordinate becomes too large). Now if 
$\abs{t-t_0}\lesssim \abs{u}^{- n}$, then $\abs{t} \leq \e \abs{u}^{-n/(h+1)} + O(\abs{u}^{-n})$, so 
if $n$ is large enough, 
$\abs{ct^{h+1}}\leq \unsur{5}\abs{\alpha u^{-n}}$ , hence  
$ \abs{\alpha u^{-n} - c t^{h+1} + O(t^{h+2})}  \asymp \abs{u}^{-n} $ and the result follows. 
\end{proof}
 
\begin{lem}\label{lem:compact}
If $r_n$ satisfies the assumption~\eqref{eq:distance} from  the previous lemma, and 
$(m_n)$ is a sequence such that $\abs{m_n  - \frac{\ln{\abs{u}}}{\ln\abs{s\inv}} n}\leq B$, then 
there exists a compact subset $L$ of  $W^s(p)\setminus \set{p}$ (for the   topology induced by the biholomorphism $W^s(p)\simeq \C$), depending only on $B$ and 
 containing the  cluster set of $(f^{-m_n}(r_n))$. 

Conversely if all cluster values of $(f^{-m_n}(r_n))$ are contained 
in some compact subset $L \subset W^s(p)\setminus \set{p}$, then 
$\abs{m_n  - \frac{\ln{\abs{u}}}{\ln\abs{s\inv}} n}\leq B(L)$.
\end{lem}
 
Note that by construction, since $r_n\in \Gamma^s_n$, 
 $f^n(r_n)$ converges to $(\alpha, 0) \in W^u(p)\setminus \set{p}$ as $n\to+\infty$.  
 
 \begin{proof}
By the previous lemma, $d(r_n, W^u(p))\asymp \abs{u} ^{-n} = \abs{s}^{n \frac{\ln \abs{u}}{\ln\abs{s\inv}}}$. So if $m_n = \frac{\ln{\abs{u}}}{\ln\abs{s\inv}} n + O(1)$, we get 
  $d(r_n, W^u(p))\asymp d(r_n, \tau)\asymp \abs{s}^{m_n}$, and 
  classical  local analysis near $p$  shows that 
  $f^{-m_n}(r_n)$ accumulates  only on a compact subset of $W^s(p)\setminus \set{p}$. One possibility 
  for this is to use the existence of a $C^1$ linearization (this is the approach in~\cite{palis_moduli}); 
  alternatively we can work 
 with a convenient normal form (see~\cite[Lem. 4.2]{df:DMM} for a very similar statement). 
  The details, as well the converse statement, are left to the reader.
 \end{proof}
 

\noindent{\bf Step 1':} uniformity of the estimates.  \commentaire{New step in the proof to give some details on the dependence on parameters, cf Exposition remark \#3. }

Here we review the proof of Step 1 to check that when $f_\lambda$ varies in the family, the estimates of 
Lemmas~\ref{lem:distance} and~\ref{lem:compact} hold locally uniformly. First, the saddle property implies that there is no resonance up to order 2, so by Proposition~\ref{prop:sternberg}
 $f_\lambda$ can be put in the form $(\star_1)$, by a change of coordinates depending holomorphically on $\lambda$. In particular  there is a local uniform bound on  $\norm{g_1}$ and $\norm{g_2}$. 
The vertical graph $\Gamma^s$ can be chosen to belong to a locally persistent horseshoe, so it varies holomorphically. Then the slope estimate depends only on the multipliers (and implicitly on $\norm{g_1}$ and $\norm{g_2}$) so it is locally uniform. The complex numbers 
 $y_0$, $\alpha$ and $c$ depend holomorphically on $\lambda$ and the $o(\cdot)$ in~\eqref{eq:polygon} is locally uniform as it derives from Rouché-type estimates as in Lemma~\ref{lem:elementary_analysis}. This implies that the implied constants in 
Lemma~\ref{lem:distance} are locally uniform as well. Finally, for Lemma~\ref{lem:compact}, we will need a uniformity in $B$ in terms of the compact subset  $L$, and conversely. First, since we work in the normalization  $(\star_1)$, the notion of a uniform compact subset $L$ in $W^s(p_\lambda)$ makes sense. For the uniformity, we can either check that the $C^1$ linearization varies continuously with $(f_\lambda)$ (there appears to be no such statement in~\cite{hartman}, but it can be deduced from its proof\footnote{Another possibility is to argue  by contradiction and    assume that $\lambda \mapsto \frac{\ln \abs{u_\lambda}}{\ln \abs{s_\lambda}}$ is non-constant. 
In this case then we can work outside a finite set of parameter with resonances, and then the existence of a $C^1$ linearization varying $C^1$ with the parameter follows from~\cite[Thm 9]{sell}.}), or resort to an easy inspection of  the proof of~\cite[Lem. 4.2]{df:DMM}.

\noindent{\bf Step 2.} Now we take advantage of the global holomorphic structure to show  
 the invariance of $\frac{\ln{\abs{u_\lambda}}}{\ln\abs{s_\lambda}} $ in   weakly stable families. 
 
By applying the automatic extension properties of plane holomorphic motions, it 
was shown in \cite[Thm 5.12 and Cor. 5.14]{tangencies} that in a   weakly stable family, the branched 
holomorphic motion of saddles extends to an equivariant normal branched holomorphic motion of $J^+
\cup J^-$, which preserves the stable and unstable manifolds of saddles. 
In some stable manifold $W^s(p_\lambda)\simeq \C$, it is  obtained by 
applying the canonical Bers-Royden  extension theorem~\cite{bers-royden}
 to the motion of $W^s(p_\lambda)\cap \jstar_\la$, viewed as a subset of $\C$.

With notation as in Step 1, 
start with some parameter $\lo$, and pick a 
point $r_n(\lo)\in W^s(p_\lambda)$ as in~\eqref{eq:distance}, 
for instance  
$r_n(\lo) = (\gamma_n(y_0(\lo)), y_0(\lo))$, corresponding to $t_0=0$. As explained above 
it admits a natural continuation $r_n(\la)$ under the branched holomorphic 
motion of $J^+$. 
The key step is the following:

\begin{lem}\label{lem:polygon} Given 
$\e< \min_{\la\in \Lambda} \frac{1}{10}\abs{\alpha(\lambda)/c(\lambda)}^{1/(h+1)}$, 
there exists a neighborhood $\Lambda_1$ of $\lo$ such that for large enough $n$, 
$r_n(\lambda)$ satisfies the assumption~\eqref{eq:distance}, that is for $\lambda\in \Lambda_1$, 
$r_n(\lambda) = (\gamma_{n, \lambda}( y_1(\la)), y_1(\la))$, with 
$\abs{y_1(\la) - y_0(\lambda)} <\e \abs{u_\lambda}^{-n/(h+1)}$. 
\end{lem}

From this point, the proof of Theorem~\ref{thm:persistent} is readily completed. Indeed, 
Lemma~\ref{lem:compact} applies so we can fix a sequence 
$(m_{j})_{j\geq 0}$   with $m_{j} = -\frac{\ln{\abs{u_0}}}{\ln\abs{s_0}} n_j + O(1)$ such that
$$f_\lo^{-m_{j}}(r_{n_j}(\lo)) \underset{j\to\infty}\longrightarrow
\zeta(\lo)\in W^s_\loc(p_0)\setminus \set{p_0}.$$ By the normality 
of the branched motion of $J^+$, 
the continuations $f_\la^{-m_{j}}(r_{n_j}(\la))$ form a normal family of graphs in $\Lambda\times \C^2$. Extracting again, we may assume that it converges to $\la\mapsto \zeta(\la)$. Since the motion of $J^+_\la$ respects stable manifolds of saddle points and is unbranched at saddle points, we infer that $\zeta(\lambda)\in W^s(p_\lambda)\setminus \set{p_\lambda}$ for every $\lambda\in \Lambda_1$, 
and we conclude 
from the converse statement of Lemma~\ref{lem:compact} that 
$\abs{m_j  - \frac{\ln{\abs{u_\lambda}}}{\ln\abs{s_\lambda\inv }} n_j}\leq B$, and finally, 
$$\frac{\ln{\abs{u_\lambda}}}{\ln\abs{s_\lambda\inv}} = \lim_{j\to\infty} \frac{m_j}{n_j} = \frac{\ln{\abs{u_0}}}{\ln\abs{s_0\inv}},$$
as desired. \qed
 
\begin{proof}[Proof of Lemma~\ref{lem:polygon}]  
By equation~\eqref{eq:polygon}, 
in the parameterization of $\Gamma_{n, \lambda}^s$ by the 
$y$-coordi\-nate, 
the  intersection points of $\Gamma_{n, \lambda}^s$ and $\Delta^u_\lambda$
 form approximately a regular $(h+1)$-gon of size 
$\asymp\abs{u_\la}^{-n/(h+1)}$.  At the parameter 
$\lo$, $y_1(\lo) = y_0(\lo)$, so 
$r_n(\lo)$ is approximately the center of this polygon. Observe that if $h=1$ this polygon 
is just a segment and $r_n(\lo)$ is close to its midpoint. \commentaire{New sentence for Major remark \#2.}
To prove
 that the estimate~\eqref{eq:distance} holds in a neighborhood of $\lo$, 
 it is enough to show that for $\la$ close to $\lo$, 
 $r_n(\la)$ remains close to the center of the polygon. For $h\geq 2$, this means that 
in the $y$-coordinate, the distances $\big|{y_1(\lambda) - \tilde y_n^{(i)}(\lambda)}\big|$ 
remain approximately equal to each other, when  varying $i$. \commentaire{continued}For $h=1$ this means that 
$$\abs{y_1(\lambda) - \tilde y_n^{(1)}(\lambda)}\simeq \abs{y_1(\lambda) - \tilde y_n^{(2)}(\lambda)}
 \simeq \frac12 \abs{\tilde y_n^{(1)}(\lambda)- \tilde y_n^{(2)}(\lambda)}.$$
For this we use a quasiconformal distortion argument. 

To be precise, we say that a map $\phi$ defined in some subset of the plane has distortion at most 
$\delta$ if for every triple of distinct points $x,y,z$, 
\begin{equation}
 \ln\lrpar{ \frac{\abs{\phi(y) - \phi(x)}}{\abs{\phi(z) - \phi(x)}} \cdot \frac{\abs{z - x}}{\abs{y-x}} } \leq \delta.
\end{equation}
With this definition, the distortion is subadditive under composition.  
We choose $\delta$ once for all such that if the distortion of 
$\big\{(\tilde y_n^{(i)}(\la))_{1\leq i\leq h+1}, y_1(\la)\big\}$ with respect to a regular
$(h+1)$-gon together with its center  (\commentaire{continued}resp. a segment with respect to its midpoint for $h=1$)
 is bounded by 
$5\delta$, then $\abs{y_1(\la) - y_0(\la)} < \e \abs{u_\la}^{-n/(h+1)}$.

As said above, 
  the points $\tilde y_n^{(i)}(\lo)$ together with $y_0(\lo)$ 
  form a regular polygon  with its center
  up to a small distortion $\delta$ for $n\geq n_0(\delta)$. Recall that by assumption 
  $\abs{y_0(\la) }<1/2$.
 Consider an  affine parameterization $ \C\ni \zeta \mapsto \psi_\la^s(\zeta)\in W^s(p)$, mapping  0 to $p$. Two such parameterizations differ by a similitude that we will not need to specify, 
 since we consider only ratios of distances. Let 
 $\tilde \zeta_n^{(i)} (\la)  = (\psi_\la^s)\inv(\tilde y_n^{(i)}(\la))$. We claim that  
  for large enough $n$ and every $\la\in \Lambda $, 
 the $\tilde \zeta_n^{(i)}(\la) $ also form a regular polygon up to distortion $\delta$.  Indeed, the map 
 $\D\ni y \mapsto (\psi^s_\la)\inv (\gamma_{n, \la}(y), y) \in \C$ is univalent, so by the Koebe distortion theorem its distortion on a disk of radius $O(\abs{u_\la}^{-n/(h+1)})$ contained  
 in $D(0, 3/4)$ is $O(\abs{u_\la}^{-n/(h+1)})$.  Thus there exists $n_1(\delta)$ so that for 
 $n\geq n_1(\delta)$ this last term is smaller than 
 $\delta$ for every $\la\in \Lambda $, and the claim follows. 
 
 Let $\zeta_1 (\la)  = (\psi_\lambda^s)\inv(r_n(\la))$,
 which is by definition the   motion of $\zeta_1(\lo)$ 
 under the Bers-Royden extension of the holomorphic motion of 
 $(\psi_\la^s)\inv (\jstar)$. We claim  that  there exists a neighborhood $\Lambda_1$ of $\lo$, 
 such that for every $\la\in \Lambda_1$ and   every 
 $n\geq n_1(\delta)$, the distortion of 
$\big\{ (\tilde \zeta_n^{(i)}(\lambda))_{1\leq i\leq r}, \zeta_1 (\la) \big\}$ with respect to  
$\big\{(\tilde \zeta_n^{(i)}(\lo))_{1\leq i\leq r}, \zeta_1 (\lo) \big\}$ is smaller than $\delta$. 
A way to see this is to 
use the fact that the Bers-Royden extension is canonical, hence automatically equivariant, so 
for $\la =\lo$ 
we can bring   the polygon to unit scale by appropriately iterating $f_\lambda$ (which is just a linear contraction in the $\zeta$-coordinate), then use 
 the uniform continuity  of the holomorphic motion at that scale (see e.g. 
\cite[Cor. 2]{bers-royden}), and then bringing it back to the original scale. 

Finally, we map 
$\big\{ (\tilde \zeta_n^{(i)}(\lambda))_{1\leq i\leq r}, \zeta_1 (\la) \big\}$ back to 
$\Gamma_{n, \la}$ by $\psi^s_\la$, which adds one more $\delta$ of distortion. Altogether,
for $\lambda\in \Lambda_1$, 
the total distortion of $\big\{(\tilde y_n^{(i)}(\la))_{1\leq i\leq h+1}, y_1(\la)\big\}$ with respect to 
a regular polygon together with its center is at most $4\delta$,  
and by the choice of $\delta$, 
the proof is complete. 
\end{proof}

\begin{rmk}
The proof carries over without essential change to the heteroclinic case,  showing that if a 
weakly stable family admits a    persistent heteroclinic tangency between
$W^s(p_1) $ and $W^u(p_2)$, then $\frac{\ln\abs{u_1}}{\ln \abs{s_2}}$ is constant. 
\end{rmk}

\subsection{Proof of Theorem~\ref{thm:newhouse_henon}}  
Thanks to  the Friedland-Milnor classification~\cite{FM},  we can assume that 
$\Lambda$ is  the space of generalized Hénon maps of some given degree sequence $(d_1, \ldots, d_k)$, that 
is, maps  of the form $h_{P_1, a_1}\circ \cdots \circ  h_{P_k, a_k}$, where $h_{P, a}(z,w) = (aw+P(z), z)$, 
so we identify $\Lambda$
 with   $\C^N\times (\C^\varstar)^k$. 
 We refer to $\C^N\times \C^k\setminus\C^N\times (\C^\varstar)^k$ 
 as the \emph{zero Jacobian locus}\commentaire{cf. Minor comment \#9}, and to 
$\widetilde \Lambda  := \C^N\times \C^k$ as the \emph{extended parameter space}. 
 
If $f_\lambda$ admits a homoclinic tangency, by  the Kupka-Smale property~\cite{BHI} (\footnote{The main theorems in the introduction of~\cite{BHI} are stated for the space of Hénon maps of degree $d$, but the authors make it clear in \S2.1 that they hold in any irreducible component of the space of generalized Hénon maps.}), 
the tangency is not locally persistent in $\Lambda$, so by Lemma~\ref{lem:higherdim}
this 
tangency persists on some local hypersurface $\mathcal T\subset \Lambda$. 

\commentaire{Paragraph added for Exposition remark \#4.} The proof will proceed in several steps. 
The first one consists in showing that we can always 
reduce to the case where $f_{\lo}$ is dissipative. 
From this point, we argue by contradiction and assume that there is a tangency parameter $\lo$ 
(associated to some primary saddle point $p = (p_\lambda)$) 
and a connected open neighborhood  $U$ of $\lo$ in $\Lambda$ 
such that 
for any $\la_1\in U$ displaying a homoclinic tangency, associated to any saddle periodic point $q$, the function $\lambda \mapsto \frac{\ln\abs{u_\la(q )}}{\ln\abs{s_\la(q )}}$ is 
constant along the corresponding local hypersurface $\mathcal T_1$. 
Note that $\mathcal T_1$ may be singular, in which case 
 the assumption means  that the constancy holds on all components of $\mathcal T_1$. 
In the second step, we show that under the contradiction hypothesis 
 the tangency locus associated to $p$ is subordinate to a real analytic foliation 
 by complex hypersurfaces
  $\mathcal F(p)$ defined in terms of the multipliers of $p$. In a third stage, we show
 that if $q$ is another saddle periodic point, the foliations $\mathcal F(p)$ and $\mathcal F(q)$ coincide, 
 so there is a canonical foliated structure in $U$. Finally, by a convenient choice of $q$ we can extend this foliation  all the way to the zero Jacobian locus to arrive at a contradiction.

Let us start with a few reductions. 
By Corollary~\ref{cor:quadratic_henon}, switching if necessary
 to another periodic point 
 (still denoted by $p$ for simplicity),  we may assume that the tangency is 
quadratic and unfolds with positive speed, in which case the corresponding hypersurface 
$\mathcal T$ is smooth 
by Lemma~\ref{lem:higherdim}. 
Abusing slightly,   we  assume for notational
convenience that $p$ is fixed\footnote{This is really an abuse because we cannot simply replace $f$ by $f^N$. 
Indeed  this would mean working 
 in  some proper subset of a component of the space of polynomial automorphisms of degree $d^N$ (i.e. the space of $N^{\rm th}$ iterates of automorphisms of degree $d$), while our argument requires to have all the component at our disposal.}  
and apply  Normalization~\ref{norm:tang}.  Reducing $U$ if needed, we may also assume that the horizontal branch $\Delta^u_\lambda$ is of degree 2. 
As in \S\ref{subs:reduction} we fix a small
horseshoe $\mathcal H$ containing $p$,      whose local stable manifolds,
denoted $W^s_\loc(x)$, for  $x\in \mathcal H$ 
(and $W^s_\loc(\mathcal H)$ for the corresponding local lamination)
 are vertical graphs in $\B$. 
Fix a countable set $(\Theta_\ell)_{\ell \geq 0}$  of vertical graphs contained in  
$W^s(p)\cap W^s_\loc(\mathcal H)$,  
   that  is dense in $W^s_\loc(\mathcal H)$. \commentaire{Fixed and explain notation, cf Minor comment \#10.}
Reducing $U$, we may assume that these objects  can be followed holomorphically  throughout $U$. 
If the horseshoe was chosen to be small enough, 
unfolding the tangency thus locally produces countably many disjoint 
hypersurfaces $\mathcal T_\ell$ in $U$ 
of generic homoclinic tangencies associated to $p$ 
(each of which corresponding to the tangency locus between 
$\Theta_{\ell, \lambda}$ and $\Delta^u_\lambda$), and  all of them are smooth.  
Since $\bigcup_\ell\mathcal T_\ell$ is Zariski dense in $U$, 
by \cite[Thm 1.4]{BHI}, moving to some other parameter $\la_0'\in \bigcup_\ell\mathcal T_\ell$ if necessary,
we may assume that the multiplier map 
  $\lambda \mapsto (s_\la(p), u_\la(p))\in \C^2$ is a submersion at $\la_0'$. 
  Rename $\la_0'$ into 
  $\lo$ and   the corresponding hypersurface by $\mathcal T$, and repeat
the previous reductions and the construction of the horseshoe 
 for this new parameter, so that we get a new sequence $(\mathcal T_\ell)$. 
 Reducing $U$ again we assume that 
  $\lambda \mapsto (s_\la(p), u_\la(p))$ is a submersion everywhere in $U$. 
  
  \medskip
  
 \noindent{\bf Step 1:} reduction to the dissipative case. \commentaire{New step for Major Remark \#3}
 
 The inverse of a composition of $k$ Hénon maps is conjugate to a composition of $k$ Hénon maps. 
Elementary calculations   show that the conjugating automorphism
 can be chosen to be linear and to locally depend holomorphically on $f$ (not globally because the choice of a $(d-1)^{\rm{th}}$ root of $\jac(f)$ is involved (see \cite[Thm. 2.6]{FM} or \cite[\S 6.1]{multipliers}). 
Therefore, if $\lambda_0\in \Lambda$ is a tangency parameter with $\abs{\jac(f_\lo)}>1$, 
there is a neighborhood $N$ of $\lo$ and 
a    holomorphic map $\iota:N\to \Lambda$    such that  
$f_{\iota(\lambda)}$ is conjugate to $f_\la\inv$. 
Using this local involution, 
it is enough to establish the theorem for the dissipative map $f_\lo\inv$. 

The argument for the conservative case is more delicate and of general interest.

 \begin{lem}\label{lem:conservative}
 If $\lambda_0$ is a parameter with a non-persistent homoclinic 
 tangency for the saddle point $p$ 
 and $\abs{\jac(f_{\lambda_0})}=1$,  then there is a tangency 
 parameter $\lambda_1$ arbitrary close to $\lambda_0$ (still associated to $p$) at which 
 $\abs{\jac(f_{\lambda_1})}< 1$.  
 \end{lem}

This lemma, together with~\cite[Cor. 4.5]{tangencies}, also shows that the second assertion of the theorem follows from the first one.

\begin{proof}
With notation as above, 
assume by contradiction that  $\bigcup_\ell \mathcal T_\ell$ is contained in the non-dissipative  
locus $\mathrm{NDis} = \set{\la\colon \abs{\jac(f_\la)} \geq  1}$.
Let also $\mathrm{Cons} =   \set{\la\colon \abs{\jac(f_\la)} =  1}$ be the conservative locus. 
Since $\lambda\mapsto (u_\la(p), s_\la(p))$ is a submersion near $\lo$, 
$\mathrm{Cons}$ is a smooth real-analytic hypersurface bounding two domains, one of them being $\mathrm{NDis}$.  

Recall that $\mathcal T$ is the locally defined  hypersurface in $\La$ 
where the initial tangency persists. 
If $\mathcal T$ is not contained in $\mathrm{Cons}$, then it meets the dissipative locus and we are done. So we may assume that $\mathcal T$ is  contained in 
$\mathrm{Cons}$, in which case the Jacobian is constant along $\mathcal T$. 

By assumption, there is an open  set of holomorphic disks $\Sigma$
through $\la_0$ along which the tangency unfolds with positive speed.  We pick one such   disk 
 that is transverse to $\mathrm{Cons}$. 
 Then near $\lo$,  $\mathrm{Cons}\cap \Sigma$ is a real analytic curve,  
 and  the countable set 
$\bigcup \mathcal T_\ell\cap \Sigma$ is contained in $\mathrm{NDis}\cap \Sigma$, 
which is one side of this curve.  
By Lemma~\ref{lem:lambda_n} and Remark~\ref{rmk:similarity}, identifying  $(\Sigma, \lo)$ with 
$(\D, 0)$, and taking the truncated pull backs of some initial graph $\Theta_{\ell_0}$, 
 we can construct a sequence $(\lambda_n)$ of   parameters in 
$ \bigcup \mathcal T_\ell\cap \Sigma$ such that 
 $\lambda_n\cong u_{\lambda_0}(p)^{-n}$. The fact that 
  $\lambda_n$ stays on one side of a smooth curve
then  forces $u_{\lo}(p)$ to be a real number. Reverting the roles of the stable and the unstable directions, we obtain that $s_{\lo}(p)$ is real as well, so 
 $\jac(f_\lo)=\pm 1$. Since $\jac{(f_\lambda)}$ is constant along $\mathcal T$, 
 we can repeat this whole reasoning for any 
  nearby parameter   $\la_1\in \mathcal T$. It follows  that  $u_{\la_1}(p)$ and $s_{\la_1}(p)$ 
  are real, so they must be constant on the hypersurface $\mathcal T$. But this is impossible
  since the locus $\set{\la: (u_\la(p), s_\la(p)) = (u_\lo(p), s_\lo(p))}$ must be of codimension 2. This 
  contradiction finishes the proof. 
\end{proof}

 \medskip
 
 \noindent{\bf Step 2:} foliating the tangency locus. 

After Step 1, we   assume  that  $\abs{\jac(f_\lo)} < 1$. Recall that from now on we argue by contradiction and assume that 
in some neighborhood of $\lo$, 
 along any hypersurface of homoclinic tangencies (associated to some saddle periodic point $q$) the 
 function $\lambda \mapsto \frac{\ln\abs{u_\la(q )}}{\ln\abs{s_\la(q )}}$ is constant. 
 
 \begin{lem} \label{lem:jacobian}
The Jacobian $\jac (f_\la)$ is not constant along $\mathcal T$. Likewise, 
$u_\lambda(p)$  and $s_\lambda(p)$ are not constant. 
\end{lem}

\begin{proof}
Indeed by assumption $\frac{\ln\abs{u_\la(p)}}{\ln\abs{s_\la(p)}} = c$ along $\mathcal T$. Since $f_\la$ is dissipative, it follows  that  $c\in (-1, 0)$. Assume by contradiction that $\jac (f_\la) \equiv j $ on 
  $\mathcal T$. Then from $u_\la s_\la = j$ we get 
$$\frac{\ln\abs{u_\la(p)}}{\ln\abs{j/u_\la(p)}} = c\text{, hence    } \ln\abs{u_\la(p)} = \frac{c}{1+c} \ln \abs{j},$$ and it follows that 
$ \abs{u_\la(p)}$ is constant along $\mathcal T$, hence so is  $u_\la(p)$. Likewise,  
  $s_\la(p)$ is constant along $\mathcal T$ because 
$u_\la s_\la = j$.  This is a contradiction
because $\lambda \mapsto (s_\la(p), u_\la(p))$ has rank 2 and $\mathcal T$ has codimension 1. 
 
Likewise, if  $u_\la(p)$ were constant along $\mathcal T$, 
 $\ln\abs{s_\la}$, and hence  $s_\la$  would be  constant 
 leading to the same contradiction. 
 \end{proof}
 
 Therefore, moving $\lo$  slightly and reducing $U$ again, 
we may assume that $u_{\la}(p)\notin\R$ for every $\lambda\in U$. 
Recall the set of disjoint hypersurfaces $\mathcal T_\ell$ in $U$ constructed above. 

\begin{lem}\label{lem:R_Zariski}
$\bigcup_\ell \mathcal T_\ell$ is $\R$-Zariski dense in $U$. 
\end{lem}

\begin{proof}\commentaire{proof completely rewritten, cf Minor comment \#11}
 Arguing as in the proof of Corollary~\ref{cor:quadratic_henon} 
and using the transversality of $\widehat{\P T\Delta^u}$ and $\widehat{\P TW^s_\loc}(p)$ and the $C^1$ closeness  of the $\Theta_\ell$ and $W^s_\loc(p)$ 
it is not difficult to see that  the tangency hypersurfaces 
$\mathcal T_\ell$ are $C^1$-close to $\mathcal T$ if the horseshoe $\mathcal H$ was chosen to be 
 close to $p$. 
Let $\Sigma$ be a   holomorphic disk transverse to $\mathcal T$ at $\lo$, along which the tangency unfolds with positive speed. From  the previous assertion, 
we may assume that  $\Sigma$ is transverse to the $\mathcal T_\ell$ as well. 

Identifying 
$(\Sigma, \lo)$ and $(\D, 0)$ and arguing exactly as in Lemma~\ref{lem:conservative}, we construct a sequence of tangency parameters $\lambda_n\in \bigcup \mathcal T_\ell\cap \Sigma$ such that 
$\lambda_n\cong u_{\lo}(p)^{-n}$. Since $u_{\lo}(p)\notin\R$, any analytic curve $Y$ containing 
$\bigcup \mathcal T_\ell\cap \Sigma$ must be singular at 0. Now, recall that the $\Theta_\ell$ are branches of 
$W^s(p)\cap W^s_\loc(\mathcal H)$, so by self-similarity of  $\mathcal H$ any branch $\Theta_\ell$ is also a limit of a sequence $\Theta_{\ell_j}$ (cf. the proof of Lemma~\ref{lem:perfect}), so we can repeat this reasoning at every point of $\bigcup \mathcal T_\ell\cap \Sigma$, and conclude $Y$ must be singular at all these points. From this we deduce that $\bigcup \mathcal T_\ell\cap \Sigma$ is $\R$-Zariski dense in $\Sigma$. 

Now, assume by way of contradiction that $\bigcup \mathcal T_\ell$ is contained in a proper 
real-analytic variety $X$, which must be of dimension $\dim_\R(\La)-1$.
From the previous paragraph we see that  
$\Sigma$ must be contained in $X$. 
But $\Sigma$  can be chosen arbitrarily in some open family, so $X=\Lambda$, 
which is a contradiction.  
\end{proof}

\commentaire{Correction for Minor comment \#12} 
If  $c= \frac{\ln\abs{u_{\la_1}(p)}}{\ln\abs{s_{\la_1}(p)}}$ for some $\lambda_1 \in U$, 
the level set 
$\frac{\ln\abs{u_\la(p)}}{\ln\abs{s_\la(p)}} = c$ is 
foliated by complex hypersurfaces of the form  $u_\la(p) = s_\la(p)^c e^{i\alpha}$,  for varying $\alpha$.
By letting $c$ vary as well,  
this defines a real-analytic foliation $\mathcal F(p)$ of $U$ by complex hypersurfaces. 
Since $\lambda \mapsto (s_\la(p), u_\la(p))$ has rank 2, $\mathcal F(p)$ is a smooth foliation. 
 Our contradiction hypothesis, together with Lemma~\ref{lem:R_Zariski}
  implies that  \emph{$\bigcup_n \mathcal T_n$ is a $\R$-Zariski dense 
set of leaves of $\mathcal F(p)$}.  

Fix now another   periodic point $q$ that is a saddle at $\lo$. 
Shifting $\lo$ slightly we may assume $\lambda \mapsto (s_\la(q), u_\la(q))$ has rank 2  at $\lo$. 
Fix a  smaller  connected open neighborhood  $V\subset U$ of $\lo$ in which $q$ remains
 a saddle and this submersion property persists. 

\begin{lem}
The foliations $\mathcal F(p)$ and $\mathcal F(q)$ coincide in $V$. 
\end{lem}

\begin{proof}  
It is enough to prove the result in some possibly smaller $V'$. 
Since $p_\lo$ and $q_\lo$ admit 
transverse heteroclinic intersections, by the inclination lemma there exists a 
countable set $(\widetilde \Theta_m)$ of disjoint vertical graphs contained in 
$W^s(q)\cap \B$, whose closure contains $W^s_\loc(\mathcal H)$, which can be followed holomorphically in some $V'$. 
Fix also a sequence of    branches $\widetilde \Delta^u_k$ of $W^u(q)\cap \B$ which converges to 
 $\Delta^u$. Restricting to sufficiently large $k$, we may assume that 
 $\widetilde \Delta^u_{k, \lambda}$ remains close to  $\Delta^u_\lambda$
  throughout $V'$, so in particular it is of horizontal degree 2 and its unique vertical tangency moves with positive speed. 
  The tangency locus 
between $\widetilde \Theta_{m,  \la}$ and $\widetilde \Delta^u_{k, \lambda}$  is then 
 a hypersurface $\widetilde {\mathcal T}_{m, k}$, 
and our contradiction hypothesis implies that each $\widetilde {\mathcal T}_{m, k}$ 
is contained in a leaf of $\mathcal F(q)$.
Let $V'$ be small enough so that the leaves of $\mathcal F(q)$ in $V'$ are contained in a foliation chart, 
so they form a  uniformly bounded family of  graphs.
 
 The key observation is that any $\Theta_{\ell}$ is the  limit of a sequence $(\widetilde \Theta_{m_j})$, uniformly in $\lambda$, that 
 is, the convergence holds for the corresponding fibered objects in $\Lambda\times \B\times \P^1$. 
 So for any sequence $k_j\to \infty$, we infer that $\widetilde {\mathcal T}_{m_j, k_j}$ converges to $\mathcal T_\ell$ in the Hausdorff (even $C^1$) topology in $V'$. 
This follows easily from  the fact that the lifts of  $\widetilde \Delta^u$ 
and $\Theta_{\ell}$ to the projectivized tangent bundle are transverse (see also the proof of Lemma~\ref{lem:R_Zariski}). 
Since the $\widetilde {\mathcal T}_{m_j, k_j}$ are contained in a foliation chart of $\mathcal F(q)$,
we conclude  that $\mathcal T_\ell$ is a leaf of $\mathcal F(q)$. Thus $\mathcal F(p)$ and $\mathcal F(q)$ coincide on a 
 $\R$-Zariski dense subset, and 
 we are done. 
\end{proof}

\noindent{\bf Step 3:} continuation and conclusion. 

Since $\lo$ belongs to the bifurcation locus,  as in Corollary~\ref{cor:quadratic_henon} we can 
select the point  $q$ so that it bifurcates to a sink in $U$. 
Consider a connected open subset $U'\subset U$ in
  which $q$ can be followed holomorphically, 
 but its unstable multiplier crosses the unit circle.  
 (Recall that by the Implicit Function Theorem,  $q$ can be locally followed holomorphically unless some eigenvalue equals 1.) Consider  a parameter $\la_1$ at which $\abs{u_{\lambda_1}(q)} = 1$ and $ {u_{\lambda_1}(q)} \neq  1$, hence 
 $u_{\lambda_1}(q) = e^{i\beta}$ for some $\beta\in \R\setminus 2\pi\Z$.  
 Then the leaf of $\mathcal F(q)$ through $\lambda_1$ is 
 $W_\beta:=
 \set{\lambda \in U',  u_{\lambda}(q) = e^{i\beta}}$. Indeed  the leaf is of the form 
  $\frac{\ln\abs{u_{\la_1}(q)}}{\ln\abs{s_{\la_1}(q)}} = c'$, but necessarily  $c'=0$. Since  
 $\mathcal F(p)$ and $\mathcal F(q)$ are real analytic and 
 coincide near $\lo$, by analytic continuation we see that $W_\beta$ is also a 
 leaf of $\mathcal F(p)$, of the form $u_\la(p) = s_\la(p)^c e^{i\alpha}$ for some $c\in(-1, 0)$. 
 
 Let 
$  \Lambda_{\mathrm {dissip}}$ be the region of the   parameter space 
made of dissipative maps.

\begin{lem}\label{lem:algebraic}
$W_\beta$  extends to  an irreducible algebraic hypersurface $\overline{W}_\beta$ in 
 the extended 
  parameter space $\widetilde \Lambda$, and \commentaire{Change here for Major remark \#4}
  the connected component $Y$ 
  of $\lambda_1$ in   $\overline W_\beta\cap \Lambda_{\mathrm{dissip}}$ intersects  the zero Jacobian locus.
\end{lem}
 
Let us admit this result for the moment and conclude the proof of the theorem. 
Let us  first show that the coincidence between 
$\mathcal F(p)$ and $\mathcal F(q)$ propagates along 
$Y$. Since these foliations are defined only in terms of the eigenvalues of $p$ and $q$, for this  it is enough to show that we can follow $p$ and $q$ as periodic points along $Y$. For $q$ this is obvious because along $W_\beta$ 
one multiplier is 
$e^{i\beta}\neq 1$ and the other one is smaller than 1 in modulus 
by dissipativity. 
For $p$, since $Y$ is a leaf of $\mathcal F(p)$ near $\lambda_1$ of the form 
$u_\la (p) = s_\la (p)^c e^{i\alpha}$, 
by (real) analytic continuation, this property persists as long as we can follow $p$. 
But    this relation implies that if one eigenvalue hits 1, then the other one has modulus 1 as well, which is impossible in the dissipative regime, and we are done. 

Now, let $\lambda \in Y$ converge 
to some parameter $\la_2\in \widetilde \Lambda$  
with Jacobian 0. For $q$, this means that the stable eigenvalue tends to 0. For $p$,   the
fact that $-1<c<0$ in the  relation   $u_\la (p) = s_\la (p)^c e^{i\alpha}$ forces 
$\abs{s_\la(p)}\inv \geq \abs{u_\la(p)}$, so $s_\la(p)$ remains the stable eigenvalue. 
Thus, as the Jacobian tends to 0, $s_\la(p)$ 
must tend to  zero, hence $\abs{u_\la(p)}$ tends to infinity, which is a contradiction 
because $f_\la$ converges to some well-defined 1-dimensional map in $\widetilde \Lambda$. 
This finishes the proof of Theorem~\ref{thm:newhouse_henon}.\qed

\begin{proof}[Proof of Lemma~\ref{lem:algebraic}]
Basic elimination theory shows that 
 $W_\beta= \set{\lambda,   u_{\lambda}(q) = e^{i\beta}}$ is defined by an algebraic condition (see e.g.~\cite[\S 2.3]{BHI}), so it 
defines an  affine algebraic hypersurface 
 in $\widetilde \Lambda$. The non-trivial fact is that 
$Y$  hits the zero Jacobian locus. 

For concreteness let us first explain the argument in
the space $\set{f_{a, c}, (a, c)\in \C^\varstar\times \C}$ of quadratic Hénon maps, where 
$f_{a,c}(z,w)  = (z^2 + c  + aw, z)$. By Lemma~\ref{lem:jacobian} 
the Jacobian (which equals  $-a$) 
is non-constant along $W_\beta$.  On the other hand, there exists 
$C$ such that the set $\set{(a, c), \  \abs{c}>C \text{ and }
\abs{a}<1}$ is contained in the horseshoe locus, and in this region 
 all periodic points are saddles.  
It follows that $\overline W_\beta\cap   \Lambda_{\mathrm {dissip}} = 
\overline W_\beta\cap \set{(a, c), \ \abs{a}<1}\subset\set{(a, c), \ \abs{c}\leq C}$. 
\commentaire{Major remark \#4, continued}
Thus $Y$ is a bounded domain in the affine algebraic curve $\overline W_\beta$, and 
$(a,c)\mapsto \jac(f_{a,c})$ is a non constant holomorphic function on $Y$, such that 
$\abs{\jac(f_{a,c})}=1$ on $\fr Y$, so its 
minimum  modulus must be 0.   (Apply the minimum principle in the normalization of $Y$, 
which  may have several irreducible components, in which case this reasoning applies to  each component.)

This argument can be transposed to the general case. 
Indeed, in \cite[Prop. 5.1]{BHI}, given 
any $f  = h_{P_1, a_1}\circ \cdots \circ  h_{P_k, a_k}\in \Lambda$, 
the authors construct an explicit algebraic 
 2-parameter family $(f_{a, c})_{(a, c)\in \C^\varstar\times \C} $, 
with $\jac (f_{a,c})\to 0$ when $a\to 0$ and such that for any $A>0$,
there exists $C>0$ such that if 
$\abs{c}>C$ and $\abs{a}<A$, all periodic points are saddles. 
So we restrict $W_\beta$ to this 2-dimensional family and argue exactly as in the quadratic case. 
\end{proof}

\appendix

\section{Weak  stability for polynomial automorphisms of $\C^2$}

Here we briefly review  the notion of weak ($\jstar$-)stability from~\cite{tangencies} 
(and further developed in~\cite{hyperbolic}). 
First, recall the usual vocabulary of complex Hénon maps:  
$K^+$ and $K^-$ are respectively the sets of bounded forward and backward orbits; 
$J^+ = \fr K^+$ and $J^-  = \fr K^-$ are the forward and backward Julia sets, and 
   $\jstar$ is  the closure of the set of saddle periodic points. 
   
Any  family of polynomial automorphisms of $\C^2$ of constant dynamical degree
is   conjugate to a family of compositions of Hénon mappings~\cite[Prop. 2.1]{tangencies}. A family of polynomial automorphisms of dynamical degree $d\geq 2$ is said to be {\em substantial} if: 
 either all its members are dissipative or for any periodic point with eigenvalues $\alpha_1$ and $\alpha_2$, 
 no relation of the form $\alpha_1^a\alpha^b_2 = c$ holds persistently in parameter space, 
 where $a$ and  $b$  are real numbers\footnote{It was inaccurately 
 stated in~\cite{tangencies} that $a$ and $b$ could be complex but the   issue is really 
  for real $a$ and $b$.}
and $\abs{c} =1$. 

 A {\em branched holomorphic motion}  is a family of holomorphic graphs over $\La$ in $\La\times \C^2$. 
It is said  \emph{normal} if these graphs form a normal family. As the  ``branched'' terminology
suggests, these graphs are allowed to intersect, while in a holomorphic motion they are not. 
  
 A substantial family $(f_\la)_{\la\in \La}$ of 
 polynomial automorphisms is said to be  {\em weakly $\jstar$-stable} if every periodic point  stays of 
 constant type (attracting, saddle, indifferent, repelling) in the family. 
 Equivalently, $(f_\la)$ is weakly $\jstar$-stable if the periodic points 
 move under a holomorphic motion. Then the   
 sets $\jstar_\la$ move under  an equivariant  branched holomorphic motion~\cite[Thm. 4.2]{tangencies}. This motion is unbranched at periodic points, heteroclinic intersections~\cite[Prop. 4.14]{tangencies}, and more generally points with some hyperbolicity~\cite{hyperbolic}. In a weakly $\jstar$-stable family, every heteroclinic or homoclinic tangency must be persistent~\cite[Prop. 4.14]{tangencies}. Note the contraposite statement: if in  a family $(f_\la)$ there is a non-persistent tangency
 at $\lo$, then some saddle point must change type near $\lo$. 
 
Using the automatic extension properties of plane holomorphic motions and the density of  
  stable and unstable manifolds of saddles, the motion of $\jstar$ can be extended to a branched holomorphic motion of $J^+\cup J^-$. It is   important   that this extended motion is normal in $\C^2$ and respects saddle points and their 
   stable and unstable manifolds, as well as  the sets $K^\pm$ and their complements~\cite[Lem. 5.10, Thm. 5.12 and Cor. 5.14]{tangencies}.  Thus we are entitled to simply call such a  family \emph{weakly stable}. 

 \bibliographystyle{acm}
\bibliography{bib-degenerate}

\end{document}